\documentclass[preprint]{elsarticle}

\usepackage{amssymb}
\usepackage{hyperref}
\hypersetup{
    colorlinks=true,
    linkcolor=blue,
    filecolor=magenta,      
    urlcolor=cyan,
    pdfpagemode=FullScreen,
    }
\usepackage{amsthm}
\usepackage{amsmath}
\usepackage{amsfonts}
\usepackage{amssymb}
\usepackage[version=4]{mhchem}
\usepackage{stmaryrd}
\usepackage{graphicx}
\usepackage[export]{adjustbox}
\graphicspath{ {./images/} }
\usepackage{bbold}
\usepackage{mathrsfs}
\usepackage[normalem]{ulem}
\newcommand{\stkout}[1]

\usepackage{color}
\ifx\red\undefined
\newcommand{\red}{\color{red}}

\newcommand{\green}{\color{green}}

\fi
\newcommand{\void}[1]{}

\begin{document}

\begin{frontmatter}



\title{Bifurcation analysis of a free boundary model of vascular tumor growth with a necrotic core and chemotaxis}

\newtheorem{theorem}{Theorem}[section]
\newtheorem{corollary}{Corollary}[theorem]
\newtheorem{lemma}[theorem]{Lemma}
\newtheorem*{verification}{Verification}
\newtheorem*{remark}{Remark}

\author[inst1]{Min-Jhe Lu}
\affiliation[inst1]{organization={Department of Mathematics, University of California at Irvine},
            city={Irvine},
            postcode={92617}, 
            state={California},
            country={United States}}

\author[inst2]{Wenrui Hao\corref{cor1}}
\ead{wxh64@psu.edu}
\cortext[cor1]{Corresponding author}
\affiliation[inst2]{organization={Department of Mathematics, Pennsylvania State University},
            city={University Park},
            postcode={16802}, 
            state={Pennsylvania},
            country={United States}}
 \author[inst3]{Bei Hu}

 \affiliation[inst3]{organization={Department of Applied and Computational Mathematics and Statistics, University of Notre Dame},
            city={Notre Dame},
            postcode={46556}, 
            state={Indiana},
            country={United States}}
            
\author[inst4]{Shuwang Li}

\affiliation[inst4]{organization={Department of Applied Mathematics, Illinois Institute of Technology},
            city={Chicago},
            postcode={60616}, 
            state={Illinois},
            country={United States}}

\begin{abstract}

A considerable number of research works has been devoted to the study of tumor models. Several biophysical factors, such as cell proliferation, apoptosis, chemotaxis, angiogenesis and necrosis, have been discovered to have an impact on the complicated biological system of tumors. An indicator of the aggressiveness of tumor development is the instability of the shape of the tumor boundary. Complex patterns of tumor morphology have been explored in \cite{lu2022nonlinear}. In this paper, we continue to carry out a bifurcation analysis on such a vascular tumor model with a controlled necrotic core and chemotaxis. This bifurcation analysis, to the parameter of cell proliferation, is built on the explicit formulas of radially symmetric steady-state solutions. By perturbing the tumor free boundary and establishing rigorous estimates of the free boundary system, 
we prove the existence of the bifurcation branches with Crandall-Rabinowitz theorem. 
The parameter of chemotaxis is found to influence the monotonicity of the bifurcation point as the mode $l$ increases both theoretically and numerically.

\end{abstract}




\end{frontmatter}


\section{Introduction}

Mathematical models by using free boundary problems have been developed to describe the tumor growth \cite{friedman2003hyperbolic,friedman2007mathematical,friedman2004free,friedman2004hierarchy,hao2016serum}. For the non-vascular, non-necrotic tumor without chemotaxis, early bifurcation result includes \cite{friedman2001existence}.
Tumor necrosis, associated with aggressiveness of the tumor growth and poor prognosis, is important for clinical intervention and  potential treatment targets related to tumor necrosis. There are several free boundary problems developed along this direction: two free boundaries are introduced to model the movement of both necrosis and tumor in \cite{hao2012bifurcation,hao2012continuation}; the existence of radially symmetric steady-state solutions for a free boundary problem with tumor necrosis has been established in \cite{cui2001analysis}; bifurcation analysis of a free boundary problem
modeling the tumor necrosis with a Robin boundary has been studied in \cite{song2021stationary}. Moreover, various nonlinear boundary conditions for free boundary models have been studied in 
\cite{zheng2019analysis,zhou2015stability}. A free boundary model of tumor growth with inhibitors is considered in \cite{wang2014bifurcation}
 and the time-delay impact on tumor growth is studied in \cite{zhao2020impact}.

Although there are many mathematical models based on free boundary problems that are used to describe the tumor necrosis, most of them focus on the dynamics of the necrotic core by treating it as a free boundary. In this paper, we consider a model studied in \cite{lu2022nonlinear} which addresses a different question: how tumor environment (\textit{e.g.,} pressure level, nutrient concentration) would be prescribed by the control of the necrotic core? 


As illustrated in Fig.\ref{fig:domain}, let $\Omega_0$ be the necrotic core, $\Omega(t)$ be the tumor tissue, $\Gamma_0$ be the controlled necrotic boundary and $\Gamma(t)$ be the tumor boundary. 

\begin{figure}
 \centering
  \includegraphics[width=0.8\textwidth]{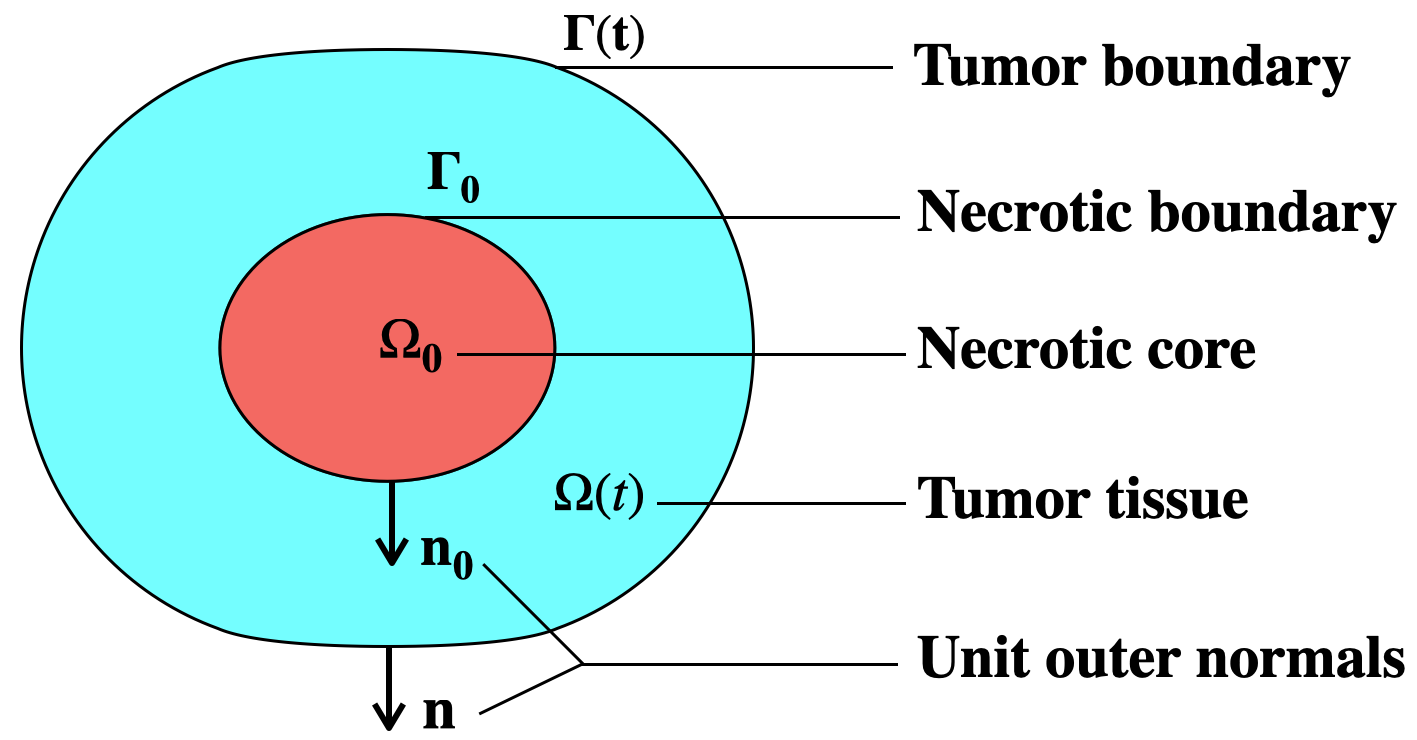}
  \caption{Illustration of the computation domain.}
  \label{fig:domain}
\end{figure}

\paragraph{\textbf{Nutrient field}}
The nutrient field in $\Omega(t)$ is governed by:
\begin{equation}\label{nut rde}
\sigma_t=D\Delta \sigma-\lambda\sigma\quad \text{in }\Omega(t),
\end{equation}
where $D,\lambda$ are the diffusion constant and 
consumption rate, respectively.
We assume Dirichlet boundary condition on the necrotic boundary: 
\begin{equation}
\sigma=\sigma^N\quad\text{on }\Gamma_0,
\end{equation}
where $\sigma^N$ is the constant nutrient level at the necrotic boundary.
We assume Robin boundary condition on the tumor boundary:
\begin{equation}
\frac{\partial\sigma}{\partial\mathbf{n}}+\beta(\sigma-\sigma^\infty)=0\quad\text{on }\Gamma(t),
\end{equation}
where $\mathbf{n}$ is the outward normal, $\bar{\sigma}$ is the constant nutrient level outside the tumor, $\beta$ is the rate of nutrient supply to the tumor, which reflects the extent of angiogenesis. 

\paragraph{\textbf{Pressure field}}
To introduce chemotaxis, the directed cell migration up gradients of nutrients, we use the generalized Darcy's law:
\begin{equation}\label{chemodarcy}
\mathbf{u}=-\mu \nabla p + \chi_{\sigma}\nabla \sigma\quad\text{in } \Omega(t),
\end{equation}
where $\mathbf{u}$ is the tumor cell velocity, $\mu$ is the cell mobility and $\chi_\sigma$ is the chemotaxis coefficient.
The mass conservation:
\begin{equation}\label{massconserv}
\nabla  \cdot \mathbf{u}=\lambda_M\frac{\sigma}{\sigma^\infty}-\lambda_A\quad\text{in }\Omega(t)
\end{equation}
where $\lambda_M$, $\lambda_A$ are the rates of mitosis (cell birth) and apoptosis (cell death), respectively.

We assume a static boundary condition on the necrotic boundary:
\begin{equation}\label{fix nec}
0=-\mu \frac{\partial p}{\partial \mathbf{n}_0}+\chi_{\sigma}\frac{\partial \sigma}{\partial \mathbf{n}_{0}}\quad\text{on }\Gamma_0,
\end{equation}
which corresponds to our assumption that the necrotic boundary is fixed.

The Laplace-Young condition is assumed on the tumor boundary:
\begin{equation}\label{eq6}
p=\gamma \left.\kappa\right|_{\Gamma(t)}\quad\text{on }\Gamma(t),
\end{equation}
where $\gamma$ is a constant representing cell-cell adhesion and $\left.\kappa\right|_{\Gamma(t)}$ is the mean curvature of the surface 
$\Gamma(t)$. In this paper we only consider the 2-space dimensional case, so $\Gamma(t)$ is actually a curve.

\paragraph{\textbf{Equation of motion}}
The equation of motion for the interface $\Gamma(t)$ is given by:
\begin{equation}\label{eq8}
V\equiv\mathbf u
\cdot 
\mathbf{n}=-\mu \left.\frac{\partial p}{\partial \mathbf{n}}\right|_{\Gamma(t)} 
+ \chi_\sigma \left.\frac{\partial \sigma}{\partial \mathbf{n}}\right|_{\Gamma(t)}\quad\text{on }\Gamma(t).
\end{equation}

In \cite{lu2022nonlinear}, the computer simulation reveals the instability of tumor free boundary through various patterns caused by several biophysical parameters in this model. More numerical experiments are investigated for Stokes-flow with the elastic membrane in \cite{lu2019nonlinear}  and for Darcy-flow with heterogeneous vasculature in \cite{lu2020complex}. In this paper, we will focus on the cell proliferation rate and carry out the bifurcation analysis to see the effect such as chemotaxis, angiogenesis, and necrosis on the bifurcation points.

This paper is structured as follows. In Section \ref{sec:Non-dimensionalization}, we give the nondimensionalized free boundary model based on nutrient and pressure fields. Next, we compute the explicit solutions of radially symmetric steady-state solutions in Section \ref{Sec:raidial}. Finally, the bifurcation analysis is studied in Section \ref{sect:bifur}.

\section{Non-dimensionalized free boundary model}
\label{sec:Non-dimensionalization}
We introduce the diffusion length $L$, the intrinsic taxis time scale $\lambda_\chi^{-1}$, and the characteristic pressure $p_s$ by:
\begin{equation}\label{eq9}
L=\sqrt{\frac{D}{\lambda}}, \quad
\lambda_\chi=\frac{\overline{\chi_{\sigma}}\sigma^{\infty}}{L^2}, \quad 
p_s=\frac{\lambda_{\chi}L^2}{\mu},
\end{equation}
where $\overline{\chi_\sigma}$ is a characteristic taxis coefficient. The length scale $L$ and the time scale ${\lambda_\chi}^{-1}$ are used to non-dimensionalize the space and time variables by $\mathbf{x}=L\widetilde{\mathbf{x}}$, $t={\lambda_\chi}^{-1}\widetilde{t}$.
Define
\begin{equation}\label{eq10}
\widetilde{\sigma}=\frac{\sigma}{\sigma^{\infty}} , \quad
\underline{\sigma}=\frac{\sigma^N}{\sigma^{\infty}} , \quad
\widetilde{p}=\frac{p}{p_s}, \quad
\widetilde{\chi_{\sigma}}=\frac{\chi_{\sigma}}{\overline{\chi_{\sigma}}}, \quad
\widetilde{\beta}=L\beta, \quad
\widetilde{\mathbf{u}}
=\frac{\mathbf{u}}{\lambda_\chi L}.
\end{equation}
Since taxis occurs more slowly than diffusion (e.g. minutes vs hours), we assume  $\lambda_{\chi} \ll \lambda$, which leads to a quasi-steady reaction-diffusion equation for the nutrient field. We remark that by the term ``taxis'', we mean taxis of tumor cells up to nutrient gradients, as embodied in Eq. \eqref{chemodarcy}. Then Eq. \eqref{nut rde} becomes $\varepsilon\widetilde\sigma_{\widetilde t}=\widetilde\Delta\widetilde\sigma-\widetilde \sigma$, where $\varepsilon=\frac{\lambda_\chi}{\lambda}\approx\frac{T_\text{diffusion}}{T_\text{taxis}}$. For the nutrient diffusion time scale $T_\text{diffusion}$, typically it can be assumed to occur in the scale of minutes, say 1 minute (see p.226 in \cite{friedman2006cancer}). For the taxis time scale $T_\text{taxis}$, we can estimate it by dividing the diameter  of the diffusion-limited tumor spheroid by the speed of migration of tumor cells up chemical gradients.  
For the tumor diameter, as summarized in \cite{grimes2014method}: ``oxygen diffusion limits are typically 100–200 $\mu m$'', and here we take the average 150 $\mu m$. For the speed of tumor migration, as summarized in \cite{roussos2011chemotaxis}: ``Some carcinoma cells with an amoeboid morphology can move at high speeds inside the tumors ($\sim 4 \mu m\text{ min}^{-1} $) ... At the other end of the range of modes of motility, ... mesenchymal migration ... (0.1--1 $\mu m\text{ min}^{-1} $)'', and here we take the average of the two types $\sim 2\mu m \text{ min}^{-1}$. Therefore, the taxis time scale can be estimated as $\frac{150 \mu m}{2\mu m \text{ min}^{-1}}=1.25$ hour.
Hence we have 
$\displaystyle\varepsilon\approx\frac{T_\text{diffusion}}{T_\text{taxis}}\approx\frac{1\  minute}{1\ hour} \ll 1$.
The dimensionless system is thus given by:
\paragraph{\textbf{Nutrient field}}
We have governing equations for the nutrient field:

\begin{equation}\label{nutrientfield}
\left\{\begin{aligned}
\widetilde{\Delta} \tilde{\sigma} &=\widetilde{\sigma} & & \text { in } \Omega(t) \\
\tilde{\sigma} &=\underline{\sigma} & & \text { on } \Gamma_{0}, \\
\frac{\widetilde{\partial} \widetilde{\sigma}}{\widetilde{\partial} \widetilde{\mathbf{n}}} &=\widetilde{\beta}(1-\widetilde{\sigma}) && \text { on } \Gamma(t)
\end{aligned}\right.
\end{equation}

where $\widetilde\beta$ (angiogenesis factor) represents the extent of angiogenesis.
\paragraph{\textbf{Pressure field}}
\begin{itemize}
    \item Non-dimensional Chemo-Darcy's law.
    \begin{equation}
\mathbf{\widetilde{u}}=-\widetilde{\nabla} \widetilde{p} + \widetilde{\chi_{\sigma}} \widetilde{\nabla} \widetilde{\sigma}\quad\text{in }\Omega(t),
\end{equation}
where $\widetilde{\chi_\sigma}$ (chemotaxis constants) represents taxis effect.
\item Conservation of tumor mass.
\begin{equation}
\widetilde{\nabla} \cdot \widetilde{\mathbf{u}}=\mathcal{P}\left(\widetilde{\sigma}- \mathcal{A}\right)\quad\text{in }\Omega(t),
\end{equation}
where $\displaystyle \mathcal{P}=\frac{\lambda_M}{\lambda_\chi}$ (proliferation rate) represents the rate of cell mitosis relative to taxis, $\displaystyle \mathcal{A}=\frac{\lambda_A}{\lambda_M}$ (apoptosis rate) represents apoptosis relative to cell mitosis.
\item Boundary conditions.
\begin{align}
\left.\frac{\widetilde\partial \widetilde p}{\widetilde\partial \widetilde{\mathbf{n}}_0}\right|_{\Gamma_0}&=\widetilde{\chi_{\sigma}}\left.\frac{\widetilde\partial\widetilde\sigma}{\widetilde\partial \widetilde{ \mathbf{n}_0}}\right|_{\Gamma_0}&\text{on }&\Gamma_0,\label{}\\
\left.\widetilde p\right|_{\Gamma(t)}&=\widetilde{\mathcal{G}}^{-1}\left.\widetilde\kappa\right|_{\Gamma(t)}&\text{on }&\Gamma(t),\label{}
\end{align}
where $\displaystyle \widetilde{\mathcal{G}}^{-1} 
=\frac{\mu\gamma}{\lambda_\chi L^3}$ represents the relative strength of cell-cell interactions (adhesion).
\end{itemize}
Therefore, we have governing equations for pressure field:
\begin{equation}\label{pressurefield}
\left\{
\begin{array}{ccc}
\begin{aligned}
-\widetilde \Delta \widetilde p&=\mathcal{P}(\widetilde\sigma-\mathcal{A}) -\widetilde\chi_\sigma\widetilde\sigma&\text{in }&\Omega(t),\\
\left.\frac{\widetilde\partial \widetilde p}{\widetilde\partial \widetilde{\mathbf{n}}_0}\right|_{\Gamma_0}&=\widetilde{\chi_{\sigma}}\left.\frac{\widetilde\partial\widetilde\sigma}{\widetilde\partial \widetilde{ \mathbf{n}_0}}\right|_{\Gamma_0}&\text{on }&\Gamma_0,\\
\left.\widetilde{p}\right|_{\Gamma(t)}&=\widetilde{\mathcal{G}}^{-1}\left.\widetilde\kappa\right|_{\Gamma(t)
}&\text{on }&\Gamma(t).
\end{aligned}
\end{array}
\right.
\end{equation}
\paragraph{\textbf{Equation of motion}}
\begin{equation}\label{eqnofmotion}
\widetilde{V}=-\left.\frac{\widetilde{\partial}\widetilde{p}}{\widetilde{\partial}\widetilde{\mathbf{n}}}\right|_{\Gamma(t)}
+\widetilde{\chi_{\sigma}}\left.\frac{\widetilde{\partial}\widetilde{\sigma}}{\widetilde{\partial}\widetilde{\mathbf{n}}}\right|_{\Gamma(t)}\quad\text{on }\Gamma(t).
\end{equation}

For brevity, we remove all ``$\widetilde{\quad}$'' in the following of this paper.

\section{Explicit formulas of radially symmetric steady-state solutions}\label{Sec:raidial}
We recall some properties of the modified bessel functions $K_{n}(x)$ and $I_{n}(x) .$ These functions form a fundamental solution set of
$$
x^{2} y^{\prime \prime}+x y^{\prime}-\left(x^{2}+n^{2}\right) y=0
$$
Furthermore,
$$
\begin{array}{ll}
I_{n+1}(x)=I_{n-1}(x)-\frac{2 n}{x} I_{n}(x), & K_{n+1}=K_{n-1}(x)+\frac{2 n}{x} K_{n}(x), n \geq 1, \\
I_{n}^{\prime}(x)=\frac{1}{2}\left[I_{n-1}(x)+I_{n+1}(x)\right], & K_{n}^{\prime}(x)=-\frac{1}{2}\left[K_{n-1}(x)+K_{n+1}(x)\right], n \geq 1, \\
I_{n}^{\prime}(x)=I_{n-1}(x)-\frac{n}{x} I_{n}(x), & K_{n}^{\prime}(x)=-K_{n-1}(x)-\frac{n}{x} K_{n}(x), n \geq 1, \\
I_{n}^{\prime}(x)=\frac{n}{x} I_{n}(x)+I_{n+1}(x), & K_{n}^{\prime}(x)=\frac{n}{x} K_{n}(x)-K_{n+1}(x), n \geq 0.
\end{array}
$$
In particular, $\forall n\ge 0$,
$$
\begin{aligned}
I_{n}(x)>0 \quad &\text { and } \quad K_{n}(x)>0,\\
I_{n}^{\prime}(x)>0 \quad &\text { and } \quad K_{n}^{\prime}(x)<0.
\end{aligned}
$$
In this paper, we will only consider the two space dimensional case. The three-dimensional case can be considered in a similar manner, except that much more computing power is needed.

We now introduce the two-dimensional polar coordinate with the radial coordinate $r$ and the angular coordinate $\theta$. Then, $\sigma$ and $p$ are functions of $(r, \theta, t)$ and the free boundary is represented as $r-R(\theta, t)=0$. From now on we drop all tildes and overbars for brevity.

Specifically, in the radially symmetric case, since $\frac{\partial \sigma}{\partial \theta}=\frac{\partial p}{\partial \theta}=0$, we have $\sigma(r, t)$ and $p(r, t) .$ Then, the free boundary is $\Gamma(t)=\{r \mid r$ $=R(t)\}$ and the fixed boundary is $\Gamma_{0}=\{r \mid r=R_0\}$. Moreover, the steady-state solutions are denoted as $\sigma(r)$ and $p(r)\ (R_0 \leq r \leq R)$ since $t \rightarrow \infty$.

\paragraph{\textbf{Radially symmetric solution of $\sigma$}} 
First, we compute the radially symmetric solution 
and have
$$
\left\{
\begin{array}{lll}
\begin{aligned}
\sigma_s^{\prime \prime}(r)+\frac{1}{r} \sigma_s^{\prime}(r)-\sigma_s(r)&=0, \\
\sigma_s(R_0)&=\underline\sigma, \\
\sigma_s^{\prime}(R)&=\beta(1-\sigma_s(R)).
\end{aligned}
\end{array}
\right.
$$
The solutions are the modified Bessel function functions of the first and second kinds, and can be written as:
\begin{equation}\label{nut_rad}
\sigma_{s}(r)=A_{1} I_{0}\left(r\right)+A_{2} K_{0}\left(r\right)
\end{equation}
Since $I_{0}^{\prime}(r)=I_{1}(r)$ and $K_{0}^{\prime}(r)=-K_{1}(r)$, we solve for $A_{1}$ and $A_{2}$ by using the boundary conditions, namely,
$$
\left\{
\begin{array}{cc}
\begin{aligned}
A_{1} I_{0}\left(R_{0}\right)+A_{2} K_{0}\left(R_{0}\right)&=\underline{\sigma},\\
A_{1} I_{1}(R)-A_{2} K_{1}(R)&=\beta\left(1-A_{1} I_{0}(R)-A_{2} K_{0}(R)\right).\label{rblin}
\end{aligned}
\end{array}
\right.
$$
Then, we have
$$
\begin{aligned}
A_{1}(\beta,\underline{\sigma}, R_0,R)
&=
\frac{\underline{\sigma}\left(K_{1}(R)-\beta K_{0}(R)\right)+\beta K_{0}(R_{0})}
{K_{0}(R_{0})\left(I_{1}(R)+\beta I_{0}(R)\right)+I_{0}(R_{0})\left(K_{1}(R)-\beta K_{0}(R)\right)},\\ 
A_{2}(\beta,\underline{\sigma},R_0,R)
&=
\frac{\underline{\sigma}\left(I_{1}(R)+\beta I_{0}(R)\right)-\beta I_{0}(R_{0})}
{K_{0}(R_{0})\left(I_{1}(R)+\beta I_{0}(R)\right)+I_{0}(R_{0})\left(K_{1}(R)-\beta K_{0}(R)\right)}.
\end{aligned}
$$
Thus
\begin{equation}\label{sigEF}
\sigma_s(r)=
\underline{\sigma}
E_{\beta}(r)
+
F_{\beta}(r),
\end{equation}
where 
$$
\begin{aligned}
E_{\beta}(r)&=
\frac{I_1(R)K_0(r)+K_{1}(R)I_0(r)
+\beta(I_0(R)K_0(r)-K_0(R)I_0(r))}
{I_1(R)K_0(R_0)+K_1(R)I_0(R_0)
+\beta(I_0(R)K_0(R_0)-K_0(R)I_0(R_0))},\\
F_{\beta}(r)&=
\frac{
\beta(
K_0(R_0)I_0(r)-I_0(R_0)K_0(r)
)
}
{I_1(R)K_0(R_0)+K_1(R)I_0(R_0)
+\beta
(K_0(R_0)I_0(R)-I_0(R_0)K_0(R))
}.
\end{aligned}
$$

\begin{lemma} \label{lemma:E,F}
For all $r\in[R_0,R]$, $E_\beta(r),F_\beta(r)$ have the following properties:
\begin{enumerate}
    \item $E_\beta(r), F_\beta(r)\in [0,1]$.
    \item $E_\beta^\prime(r)\leq 0, F_\beta^\prime(r)\geq 0$.
\end{enumerate}
\end{lemma}
\begin{proof}
Let $r\in [R_0,R]$. Note that $E_\beta(R_0)=1$ and $F_\beta(R_0)=0$. Consider
$$
\begin{aligned}
I_{0}(R) K_{0}(r)-K_{0}(R) I_{0}(r)\ge0
&\Leftrightarrow
I_{0}(R) K_{0}(r)\ge K_{0}(R) I_{0}(r),
\end{aligned}
$$
which is true since $I_{0}(R)\ge I_{0}(r)$ and $K_{0}(r)\ge K_{0}(R)$, thus $E_\beta(r)\ge 0$. (Recall $I_l(r)>0$ and $K_l(r)>0$ for all $l\ge 0$.)
Similarly, 
since $K_{0}(r)\le K_{0}(R_0)$ and $I_{0}(r)\le I_{0}(R_0)$ imply $K_{0}(R_0) I_{0}(r)-I_{0}(R_0) K_{0}(r)\ge0$, we have $F_\beta(r)\ge 0$.
Next, we show $E_\beta(r)\leq 1$ by considering
$$
I_{1}(R) K_{0}^\prime(r)+K_{1}(R) I_{0}^\prime(r)\le0
\Leftrightarrow
 K_{1}(R)I_{1}(r)\le I_{1}(R)K_{1}(r),
$$
which is true since $I_{1}(R)\ge I_{1}(r)$ and $K_{1}(r)\ge K_{1}(R)$. Moreover, since
$I_{0}(R) K_{0}^\prime(r)-K_{0}(R) I_{0}^\prime(r)=-(I_{0}(R) K_{1}(r)+K_{0}(R) I_{1}(r))\leq 0$, we have $E_{\beta}^\prime(r)\leq 0$ and with $E_\beta(R_0)=1$ we obtain $E_{\beta}(r)\leq 1$. Similarly, since $K_0(R_0)I_0^\prime(r)-I_0(R_0)K_0^\prime(r)=K_0(R_0)I_1(r)+I_0(R_0)K_1(r)\geq 0$, we have $F_\beta^\prime(r)\geq 0$ and with $F_\beta(R)\leq 1$ we obtain $F_\beta(r)\leq 1$. These results are also verified numerically in Fig. \ref{fig:EF}.
\end{proof}

\begin{theorem}
$\sigma_{s}(r)\in[0,1],\forall r\in[R_0,R]$.
\end{theorem}
\begin{proof}
It is the consequence of $\underline{\sigma}<1$,   \eqref{sigEF} and Lemma \ref{lemma:E,F}.
\end{proof}

\begin{remark}
By weak maximum principle, we also have $0\leq\sigma_{s}(r) \leq 1$ for $R_0 \leq r \leq R$. 
\end{remark}

\begin{figure}
    \centering
    \includegraphics[width=\textwidth]{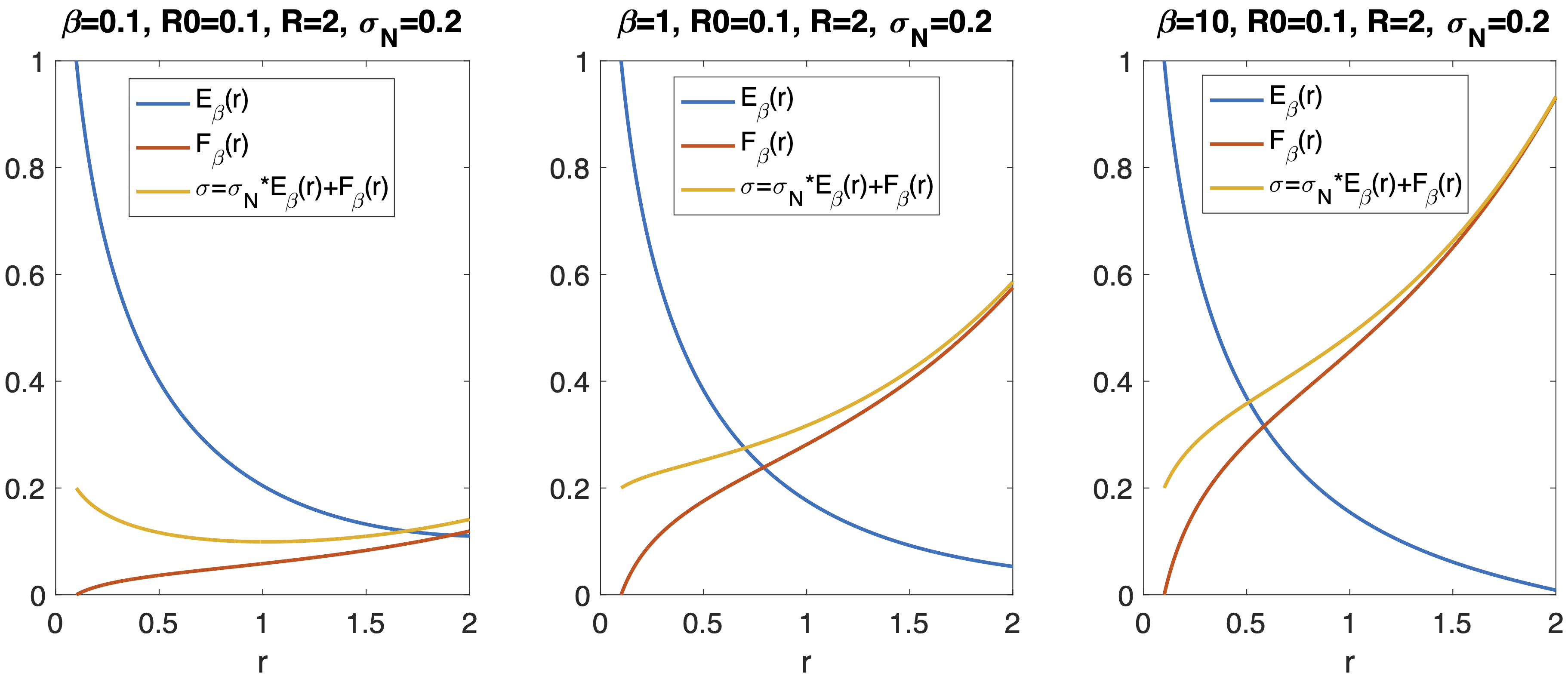}
    \caption{A numerical validation of Lemma \ref{lemma:E,F} by varying $\beta=0.1,1,10$ and fixing $l=2$.}
    \label{fig:EF}
\end{figure}

\noindent
Note that $\forall r\in(R_0,R]$,
$$
\begin{aligned}
\lim_{\beta\rightarrow 0}\sigma_s(r)
&
=
\underline{\sigma}
E_0(r)
,\\ 
\lim_{\beta\rightarrow \infty}\sigma_s(r)
&=
\underline{\sigma}
E_\infty(r)
+F_\infty(r).
\end{aligned}
$$
where
$$
\begin{aligned}
E_0(r)&=\frac{
I_1(R)K_0(r)+K_{1}(R)I_0(r)}
{I_1(R)K_0(R_0)+K_1(R)I_0(R_0)},\\
E_\infty(r)&=
\frac{I_0(R)K_0(r)-K_{0}(R)I_0(r)}
{I_0(R)K_0(R_0)-K_0(R)I_0(R_0)},\\
F_\infty(r)&=
\frac{K_0(R_0)I_0(r)-I_0(R_0)K_0(r)
}
{K_0(R_0)I_0(R)-I_0(R_0)K_0(R)}.
\end{aligned}
$$
Since $\displaystyle\lim_{r\rightarrow 0}I_0(r)=1,\lim_{r\rightarrow 0}K_0(r)=\infty$, we have 
\begin{align}
\label{nutlim}
\lim_{\substack{\beta\rightarrow \infty\\R_0\rightarrow 0}}\sigma_s(r)
&=
\frac{I_0(r)}{I_0(R)},\\
\lim_{\substack{\beta\rightarrow \infty\\R_0\rightarrow 0}}\sigma^\prime_s(r)
&=
\frac{I_1(r)}{I_0(R)}.
\end{align}

\paragraph{\textbf{Radially symmetric solution of $p$}} 
By rewriting the first equation in \eqref{pressurefield} as $\Delta(p+\left (\mathcal{P}-\chi_{\sigma}\right) \sigma)=\mathcal{P}\mathcal{A}$, we have
\begin{equation}\label{ps_rad}
p_{s}(r)=- (\mathcal{P}-\chi_{\sigma})\sigma_{s}(r)+C_{1}+C_{2} \ln r+\frac{\mathcal{P}\mathcal{A}}{4}  r^{2} .
\end{equation}
The boundary conditions in radially symmetric case become
$$
p_{s}(R)=\mathcal{G}^{-1} \frac{1}{R},\quad p_s^\prime(R_0)=\chi_\sigma \sigma_s^\prime(R_0)\text{, and } p_s^\prime(R)=\chi_\sigma \sigma_s^\prime(R),
$$
which are used to determine $C_{1}, C_{2}$,
and $\mathcal{A}$,
namely,
$$
\begin{aligned}
C_{1}&=\mathcal{G}^{-1} \frac{1}{R}+(\mathcal{P}-\chi_{\sigma})\sigma_s(R)
-C_2 \ln(R)
-\frac{\mathcal{P} \mathcal{A}R^2}{4}, \nonumber \\
C_{2}&=\mathcal{P}\sigma_s^\prime(R_0) R_{0}-\frac{\mathcal{P} \mathcal{A}R_{0}^{2} }{2} ,
\end{aligned}
$$
where
\begin{equation}\label{apopt}
\mathcal{A}=\frac{2}{R^2-R_0^2}\left(R\sigma_s^\prime(R)-R_0\sigma_s^\prime(R_0) \right).
\end{equation}
Thus
\begin{align}
    p_s(r)=&\ \notag 
    \mathcal{G}^{-1} \frac{1}{R}
     +(\mathcal{P}-\chi_\sigma)(\sigma_s(R)-\sigma_s(r))\\
    &\ +\left(\mathcal{P}\sigma_s^\prime(R_0)R_0-\frac{\mathcal{P}\mathcal{A}R_0^2}{2}\right)\ln\left(\frac{r}{R}\right)
     -\frac{\mathcal{P}\mathcal{A}}{4}(R^2-r^2),\\
     p^\prime_s(r)=&\frac{\mathcal{P}\mathcal{A}}{2}r-(\mathcal{P}-\chi_{\sigma})\sigma^\prime(r)
     +\left(\mathcal{P}\sigma^{\prime}_s(R_0)-\frac{\mathcal{P}\mathcal{A}R_0}{2} \right)\frac{R_0}{R}.
\end{align}
Note that
$$
\begin{aligned}
\lim_{\beta\rightarrow \infty, R_0\rightarrow 0}p_s(r)&=
\mathcal{G}^{-1} \frac{1}{R}
+(\mathcal{P}-\chi_\sigma)\left(1-\frac{I_0(r)}{I_1(R)}\right)
-\frac{\mathcal{P}\mathcal{A}}{4}(R^2-r^2),\\
\lim_{\beta\rightarrow \infty, R_0\rightarrow 0}p^\prime_s(r)&=
\frac{\mathcal{P}\mathcal{A}}{2}r-(\mathcal{P}-\chi_\sigma)\frac{I_0(r)}{I_1(R)}
.
\end{aligned}
$$
For any given $\mathcal{A}$, we compute $R$ by solving $\eqref{apopt}$. Therefore, the existence of $R$ is critical for our model. In order to prove the existence, we solve $\mathcal{A}$ for any given $R$ and have the following theorem.

\begin{theorem}
For any given $\mathcal{P},\chi_\sigma,\beta,R_0$ and $R>0$, there exists a unique $\mathcal{A}>0$ such that a stationary solution $\left(\sigma_{s}, p_{s}\right)$ is given by $\eqref{nut_rad}$ and $\eqref{ps_rad}$.
\end{theorem}

\begin{proof}
For any given $R$, it is obvious that $\mathcal{A}$ is uniquely determined by \eqref{apopt}. Next, we prove $\mathcal{A}>0$ by letting
$$
f(r)=r \sigma_{s}^{\prime}(r) .
$$
Since
$$
f^\prime(r)=r \sigma_{s}^{\prime \prime}(r)+\sigma_{s}^{\prime}(r)= r \sigma_{s}(r) \geq 0,
$$
we have
$$
R \sigma_{s}^{\prime}(R)=f(R) \geq f(R_0)=R_0 \sigma_{s}^{\prime}(R_0),
$$
which implies
$$
\mathcal{A}=\frac{2}{R^2-R_0^2}\left(R\sigma_s^\prime(R)-R_0\sigma_s^\prime(R_0) \right) \geq 0.
$$
\end{proof}

\section{Bifurcation analysis}
\label{sect:bifur}
\subsection{The linearized system}
First, we derive the linearized system of equations \eqref{nutrientfield},\eqref{pressurefield} and \eqref{eqnofmotion} with a perturbed domain $\Omega_{\varepsilon}$ to $\Omega$, namely, $\Gamma_{\varepsilon}=\left\{r \mid r=R+\varepsilon R_{1}(\theta)\right\}$,

\begin{equation}\label{linpde}
\left\{\begin{array}{lll}
\begin{aligned}
\Delta \sigma&=\sigma, &x \in \Omega_{\varepsilon}, \\ 
\sigma&=\underline\sigma, &x \in \Gamma_{0}, \\
\frac{\partial \sigma}{\partial \mathbf{n}}&=\beta(1-\sigma), &x \in \Gamma_{\varepsilon}, \\ 
-\Delta p&=(\mathcal{P}-\chi_\sigma)\sigma-\mathcal{P}\mathcal{A}, &x \in \Omega_{\varepsilon}, \\
\frac{\partial p}{\partial \mathbf{n}_0}&=\chi_\sigma \frac{\partial \sigma}{\partial\mathbf{n}_0}, &x \in \Gamma_{0}, \\ 
p&=\mathcal{G}^{-1} \kappa, &x \in \Gamma_{\varepsilon}.
\end{aligned}
\end{array}\right.       
\end{equation}

By defining the following nonlinear function $F$ based on the free boundary condition:
\begin{equation}\label{bifurcation}
  F\left(\varepsilon R_{1}, \mathcal{P}\right)=-\left.\frac{\partial p}{\partial r}\right|_{\Gamma_{\varepsilon}}+\chi_{\sigma}\left.\frac{\partial \sigma}{\partial r}\right|_{\Gamma_{\varepsilon}},  
\end{equation}
we conclude that $R_{1}(\theta)$ induces a stationary solution if and only if $F\left(R_{1}, \mathcal{P}\right)=0$. Then, we consider the solution of $(11),(\sigma, p)$, up to the second order of $\varepsilon$, in a formal expansion:

\begin{equation}\label{expansions}
\begin{aligned}
\sigma(r, \theta) &=\sigma_{s}(r)+\varepsilon \sigma_{1}(r, \theta)+\mathcal{O}\left(\varepsilon^{2}\right), \\
p(r, \theta) &=p_{s}(r)+\varepsilon p_{1}(r, \theta)+\mathcal{O}\left(\varepsilon^{2}\right) .
\end{aligned}
\end{equation}

$\sigma$ on $\Gamma_{\varepsilon}$ is then given by
$$
\begin{aligned}
\left.\sigma(r, \theta)\right|_{\Gamma_{\varepsilon}}&=\sigma\left(R+\varepsilon R_{1}, \theta\right) \\
&=\sigma_{s}\left(R+\varepsilon R_{1}\right)+\varepsilon \sigma_{1}\left(R+\varepsilon R_{1}, \theta\right)+\mathcal{O}\left(\varepsilon^{2}\right)\\
&=\sigma_{s}(R)+\varepsilon R_{1} \sigma_s^\prime(R)+\varepsilon \left.\sigma_{1}\right|_{r=R}+\mathcal{O}\left(\varepsilon^{2}\right)\\
&=\sigma_{s}(R)+\varepsilon\left( R_{1} \sigma_s^\prime(R)+ \left.\sigma_{1}\right|_{r=R}\right)+\mathcal{O}\left(\varepsilon^{2}\right).
\end{aligned}
$$
Similarly,
$$
\left.p(r, \theta)\right|_{\Gamma_{\varepsilon}}
=p_{s}(R)+\varepsilon\left( R_{1} p_s^\prime(R)+ \left.p_{1}\right|_{r=R}\right)+\mathcal{O}\left(\varepsilon^{2}\right).
$$
Thus, the boundary condition on $\Gamma_\varepsilon$ for $\sigma$ becomes
$$\begin{aligned}
\sigma_s^\prime(R)&+\varepsilon\left( R_1\sigma_s^{\prime\prime}(R)+\left.\frac{\partial\sigma_1}{\partial r}\right|_{r=R} \right)\\
&=\beta\left(1- 
\left(\sigma_{s}(R)+\varepsilon\left( R_{1} \sigma_s^\prime(R)+ \left.\sigma_{1}\right|_{r=R}\right) \right)\right)
+\mathcal{O}\left(\varepsilon^{2}\right).
\end{aligned}
$$
Since the mean curvature is given by
$$
\kappa=\frac{r^{2}+2 r_{\theta}^{2}-r r_{\theta \theta}}{\left(r^{2}+r_{\theta}^{2}\right)^{3 / 2}},
$$
the linearization of $\kappa$ becomes
$$
\begin{aligned}
\left.\kappa\right|_{\Gamma_{\varepsilon}}=& \frac{\left(R+\varepsilon R_{1}\right)^{2}+2\left(R_{\theta}+\varepsilon R_{1 \theta}\right)^{2}}{\left(\left(R+\varepsilon R_{1}\right)^{2}+\left(R_{\theta}+\varepsilon R_{1 \theta}\right)^{2}\right)^{3 / 2}} \\
&-\frac{\left(R+\varepsilon R_{1}\right)\left(R_{\theta \theta}+\varepsilon R_{1 \theta \theta}\right)}{\left(\left(R+\varepsilon R_{1}\right)^{2}+\left(R_{\theta}+\varepsilon R_{1 \theta}\right)^{2}\right)^{3 / 2}} \\
=& \kappa_{0}+\varepsilon \kappa_{1}+\mathcal{O}\left(\varepsilon^{2}\right),
\end{aligned}
$$
where
$$
\kappa_{0}=\frac{R^{2}+2 R_{\theta}^{2}-R R_{\theta \theta}}{\left(R^{2}+R_{\theta}^{2}\right)^{3 / 2}}=\frac{R^2}{R^3}=\frac{1}{R},
$$
and
$$
\begin{aligned}
\kappa_{1}=&\left(\frac{2 R-R_{\theta \theta}}{\left(R_{\theta}^{2}+R^{2}\right)^{3 / 2}}-\frac{3}{2} \frac{\left(R^{2}+2 R_{\theta}^{2}-R R_{\theta \theta}\right) 2 R}{\left(R_{\theta}^{2}+R^{2}\right)^{5 / 2}}\right) R_{1} \\
&+\left(\frac{4 R_{\theta}}{\left(R_{\theta}^{2}+R^{2}\right)^{3 / 2}}-\frac{3}{2} \frac{\left(R^{2}+2 R_{\theta}^{2}-R R_{\theta \theta}\right) 2 R_{\theta}}{\left(R_{\theta}^{2}+R^{2}\right)^{5 / 2}}\right) R_{1 \theta} \\
&-\frac{R}{\left(R_{\theta}^{2}+R^{2}\right)^{3 / 2}} R_{1 \theta \theta}\\
=&\left(\frac{2R}{R^3}-\frac{3R^3}{R^5}\right)R_{1}
-\frac{R}{R^3}R_{1\theta\theta}\\
=&-\frac{R_1+R_{1\theta\theta}}{R^2}.
\end{aligned}
$$
After dropping the higher-order terms, we obtain the linearized system below,
\begin{equation}\label{linearized}
\left\{\begin{aligned}
\Delta \sigma_{1}&= \sigma_{1} \quad \text { in } \Omega, \\
\left.\sigma_{1}\right|_{r=R_0}&=0, \\
R_1\sigma_s^{\prime\prime}(R)+\left.\frac{\partial\sigma_1}{\partial r}\right|_{r=R}&=-\beta
\left( R_{1} \sigma_s^\prime(R)+ \left.\sigma_{1}\right|_{r=R} \right), \\
\Delta p_{1}&=-(\mathcal{P}-\chi_\sigma) \sigma_{1} \quad \text { in } \Omega, \\
\left.\frac{\partial p_{1}}{\partial r}\right|_{r=R_0}
&=\chi_\sigma\left.\frac{\partial \sigma_{1}}{\partial r}\right|_{r=R_0}, \\
R_{1} p_s^\prime(R)+ \left.p_{1}\right|_{r=R}&=-\frac{\mathcal{G}^{-1}}{R^{2}}\left(R_{1}+R_{1 \theta \theta}\right).
\end{aligned}\right.
\end{equation}
By separation of variables, $\sigma_{1}(r, \theta)= Q_{l}(r)R_1(\theta)$, and assuming $R_{1}(\theta)=\cos (l \theta)$, we have
\begin{equation}
Q_{l}(r)={B}_{1} I_{l}\left(r\right)+{B}_{2} K_{l}\left(r\right),    
\end{equation}
which satisfies 
\begin{equation}\label{Ql}
    \left\{
    \begin{aligned}
    r^{2} Q_{l}^{\prime \prime}+r Q_{l}^{\prime}-\left(r^{2}+l^{2}\right) Q_{l}&=0,\\
    Q_{l}(R_0)&=0,\\
    \sigma_s^{\prime\prime}(R)+Q_l^{\prime}(R)&=-\beta
\left(\sigma_{s}^{\prime}(R)+Q_{l}(R) \right).
    \end{aligned}
    \right.
\end{equation}
Therefore $B_1$ and $B_2$ satisfy
$$
\left\{
\begin{aligned}
B_{1} I_{l}\left( R_{0}\right)+B_{2} K_{l}\left( R_{0}\right)&=0,\\
B_{1}I^\prime_{l}(R)+B_{2}K^\prime_l(R)
&=-\sigma^{\prime\prime}_s(R)-\beta\left(\sigma_{s}^{\prime}(R)+B_{1} I_{l}\left( R\right)+B_{2} K_{l}\left( R\right) \right),
\end{aligned}
\right.
$$
then we have
$$
B_1=K_l(R_0)B(\beta,R_0,R), \ 
B_2=-I_l(R_0)B(\beta,R_0,R),
$$
where
$$
B(\beta,R_0,R)=\frac{\sigma_s^{\prime\prime}(R)+\beta\sigma_s^\prime(R)}{I_l(R_0)(K_l^\prime(R)+\beta K_l(R))-K_l(R_0)(I_l^\prime(R)+\beta I_l(R))}.
$$
Thus
$$
Q_l(r)=-(\sigma_s^{\prime\prime}(R)+\beta\sigma_s^\prime(R))
G_\beta(r),
$$
where
$$
G_\beta(r)=\frac{K_l(R_0)I_l(r)-I_l(R_0)K_l(r)}
{K_l(R_0)I_l^\prime(R)-I_l(R_0)K_l^\prime(R)+\beta(K_l(R_0)I_l(R)-I_l(R_0)K_l(R))}.
$$
Note that
\begin{align}
\lim_{\beta\rightarrow 0}Q_l(r)
&=-\underline\sigma E_0^{\prime\prime}(R)G_0(r),\\ 
\lim_{\beta\rightarrow \infty}Q_l(r)
&=-(\underline\sigma E_\infty^{\prime}(R)+F_\infty^{\prime}(R))G_\infty(r),
\end{align}
where
$$
\begin{aligned}
G_0(r)
&=\frac{K_l(R_0)I_l(r)-I_l(R_0)K_l(r)}{K_l(R_0)I_l^\prime(R)-I_l(R_0)K_l^\prime(R)},\\ 
G_\infty(r)
&=\frac{K_l(R_0)I_l(r)-I_l(R_0)K_l(r)}{K_l(R_0)I_l(R)-I_l(R_0)K_l(R)}.
\end{aligned}
$$
Since $\displaystyle\lim_{r\rightarrow 0}I_l(r)=0,\lim_{r\rightarrow 0}K_l(r)=\infty,\forall l\ge 2$, and \eqref{nutlim} , we have
\begin{align}\label{Qllim}
\lim_{\substack{\beta\rightarrow \infty\\R_0\rightarrow 0}}
Q_l(r)&=-\frac{I_1(R)}{I_0(R)}\frac{I_l(r)}{I_l(R)},  \\
\lim_{\substack{\beta\rightarrow\infty\\R_0\rightarrow 0}}
Q_l^\prime(r)&=-\frac{I_1(R)}{I_0(R)}\frac{I_l^\prime(r)}{I_l(R)}.
\label{Qlplim}
\end{align}

\begin{figure}
    \centering
    \includegraphics[width=\textwidth]{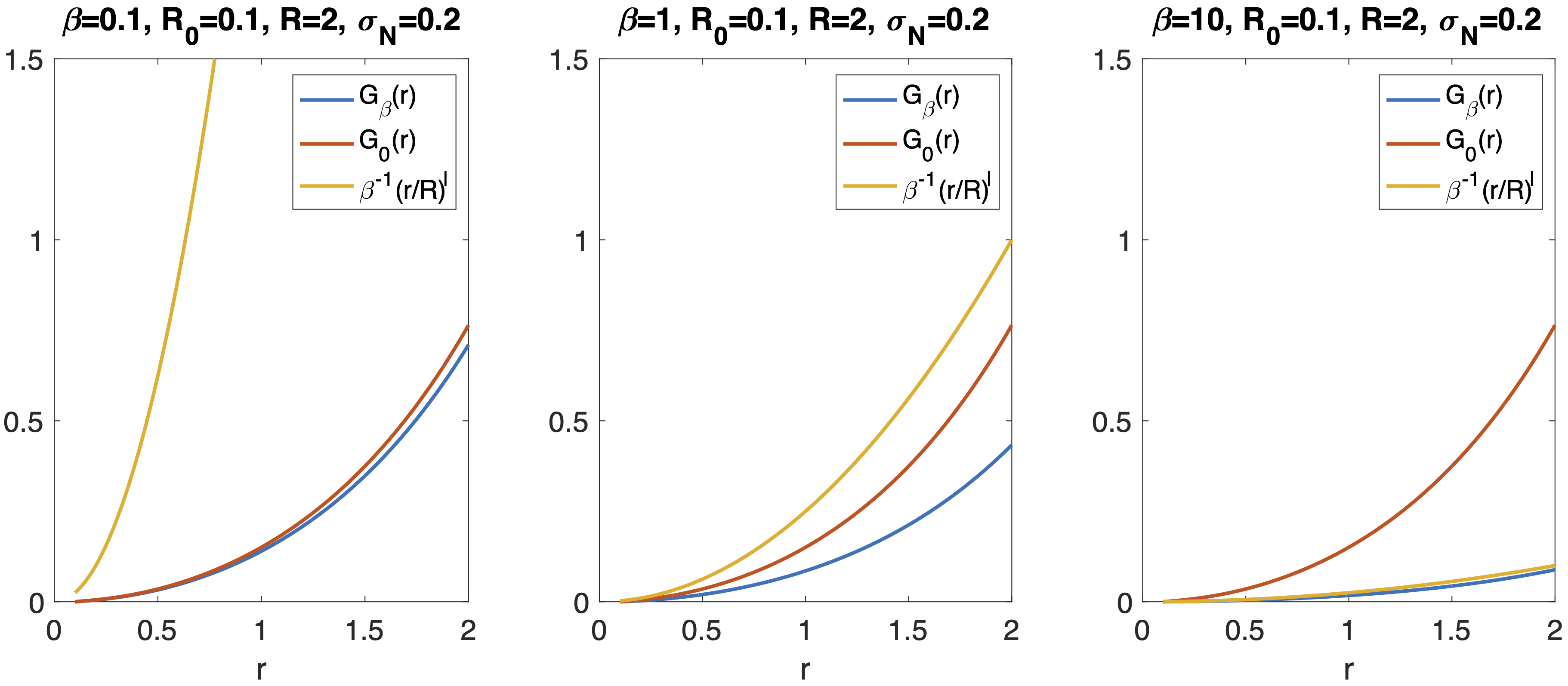}
    \caption{A numerical validation of Lemma \ref{lemma:G} by varying $\beta=0.1,1,10$ and fixing $l=2$.}
    \label{fig:G}
\end{figure}

\begin{lemma}\label{lemma:G}
For all $r\in [R_0,R], G_\beta(r)$ has the following properties:
\begin{enumerate}
    \item $G_\beta^\prime(r)\geq 0$.
    \item $G_\beta(r)\in \left[0,\min\left\{
G_0(R),\frac{1}{\beta}
\right\}\right]$.
    \item $0\le G_\beta(r)\le \min\left\{G_0(r),\frac{1}{\beta}\left(\frac{r}{R}\right)^l\right\}$.
\end{enumerate}
\end{lemma}
\begin{proof}
Let $r\in[R_0,R]$. To show (1), note that the denominator of $G_\beta(r)$ is positive since $K_l^\prime(r)<0$, $I_l^\prime(r)>0$, $0<I_l(R_0)<I_l(R)$ and $0<K_l(R)<K_l(R_0)$. Thus $ G_\beta^\prime(r)\ge 0
\Leftrightarrow K_l(R_0)I_l^\prime(r)-I_l(R_0)K_l^\prime(r)\ge 0,$ which holds true. Now, consider $G_\beta(R_0)=0$ and
$$
\begin{aligned}
G_\beta(R)
&=\frac{K_l(R_0)I_l(R)-I_l(R_0)K_l(R)}
{K_l(R_0)I_l^\prime(R)-I_l(R_0)K_l^\prime(R)+\beta(K_l(R_0)I_l(R)-I_l(R_0)K_l(R))}\\
&\le
\min\left\{G_0(R)
,\frac{1}{\beta}G_\infty(R)
\right\},
\end{aligned}
$$
combining with (1) and $G_\infty(R)=1$ we have (2).
To show (3), define $G(\mathbf x):=G_\infty(|\mathbf x|)$ and $H(\mathbf x):=\left(\frac{|\mathbf x|-R_0}{R-R_0}\right)^l$ which satisfy
$$
\left\{
\begin{aligned}
-\Delta G +\left(\frac{l^2}{r^2}+1\right)G&=0,\\
\left.G\right|_{r=R_0}&=0,\\
\left.G\right|_{r=R}&=1,
\end{aligned}
\right.
$$
and
$$
\left\{
\begin{aligned}
-\Delta H +\left(\frac{l^2}{r^2}\right)H&=0,\\
\left.H\right|_{r=R_0}&=0,\\
\left.H\right|_{r=R}&=1,
\end{aligned}
\right.
$$
respectively.
Then $G_\infty(r)\le\left(\frac{r-R_0}{R-R_0}\right)^l$ by comparison principle. Note that $\left(\frac{r-R_0}{R-R_0}\right)^l\le\left(\frac{r}{R}\right)^l$. Hence $G_\beta(r)\le\min\left\{G_0(r),\frac{1}{\beta}G_\infty(r)\right\}\le\min\left\{G_0(r),\frac{1}{\beta}\left(\frac{r}{R}\right)^l\right\}$. These results are also verified numerically in Fig. \ref{fig:G}.
\end{proof}



\begin{lemma}\label{lemma:GG'}
For all $r\in(R_0,R]$, $\displaystyle a_l(r):=\frac{G_{\beta}^{\prime}(r;l)}{G_{\beta}(r;l)}-\frac{l}{r}$ satisfies
\begin{enumerate}
    \item $a_l(r)>0$ for all $l$.
    \item $a_l(r)$ is a decreasing sequence in $l$.
\end{enumerate}
\end{lemma}
\begin{proof}
Fix any $s\in(R_0,R]$, define $$\displaystyle\psi(r)=\frac{r}{s}G_\beta(r;l)-\frac{G_\beta(s;l)}{G_\beta(s;l+1)}G_\beta(r;l+1),$$
which satisfies
$$
\left\{
\begin{aligned}
-\Delta \psi +\left(\frac{(l+1)^2}{r^2}+1\right)\psi&=\frac{2}{s}\left(\frac{l}{r}G_\beta(r;l)-G_\beta^{\prime}(r;l)\right),\\
\left.\psi\right|_{r=R_0}&=0,\\
\left.\psi\right|_{r=s}&=0.
\end{aligned}
\right.
$$
Note that
$$
\begin{aligned}
&\frac{l}{r}G_\beta(r;l)-G_\beta^{\prime}(r;l)\\&=-\frac{K_l(R_0)I_{l+1}(r)+I_l(R_0)K_{l+1}(r)}
{K_l(R_0)I_l^\prime(R)-I_l(R_0)K_l^\prime(R)+\beta(K_l(R_0)I_l(R)-I_l(R_0)K_l(R))}<0,
\end{aligned}
$$
which implies (1) using Lemma $\ref{lemma:G}$. Then using maximum principle we have $\psi(r)<0$ for $r\in(R_0,s)$. In particular, $\left.\psi^\prime(r)\right|_{r=s}>0$, which implies
$$
\frac{G_\beta^{\prime}(s;l)}{G_\beta(s;l)}-\frac{l}{s}
>
\frac{G_\beta^{\prime}(s;l+1)}{G_\beta(s;l+1)}-\frac{l+1}{s},
$$
thus we have (2) since $s\in(R_0,R]$ is arbitrarily fixed.
\end{proof}
\begin{lemma}\label{lemma:G'}For all $r\in[R_0,R]$, $\displaystyle b_l(r):=G_{\beta}^{\prime}(r;l)$ satisfies
\begin{enumerate}
    \item $b_l(r)>0$ for all $l$.
    \item $b_l(R_0)$ is a decreasing sequence in $l$.
    \item $b_l(R)$ is a increasing sequence in $l$.
\end{enumerate}

\end{lemma}
\begin{proof}
By Lemma \ref{lemma:G} we have  $0<G_\beta(r)<\frac{1}{\beta}\left(\frac{r}{R}\right)^{l}$ for $r\in(R_0,R)$, 
and $G_\beta^{\prime}(R_0;l)>0.$
Let $F=\frac{\partial G_\beta}{\partial l}$, then
$$
\left\{\begin{aligned}
-\Delta F+\left(\frac{l^{2}}{r^{2}}+1\right) F &=-\frac{2 l}{r^{2}} G_\beta, 
\\
\left.F\right|_{r=R_0} &=0, 
\\
\left.F\right|_{r=R} &= \left.\frac{\partial G_\beta}{\partial l}\right|_{r=R},
\end{aligned}\right.
$$
and
$$
\left.\frac{\partial F}{\partial r}\right|_{r=R}
=
\left.\frac{\partial }{\partial r}
\frac{\partial G_\beta}{\partial l}\right|_{r=R}
=
\frac{\partial }{\partial l}
G_\beta^{\prime}(R)
=
\frac{\partial }{\partial l}(-\beta G_\beta(R)+1)
=\left.-\beta\frac{\partial G_\beta}{\partial l}\right|_{r=R},
$$
where we used the expression of $Q_l(r)=-(\sigma_s^{\prime\prime}(R)+\beta\sigma_s^\prime(R))
G_\beta(r)$ and its derivative with respect to $r$ at $r=R$, $Q_l^\prime(R)=-(\sigma_s^{\prime\prime}(R)+\beta\sigma_s^\prime(R))
G_\beta^\prime(R)$ and the boundary condition of $Q_l(R)$ in $(31)$ to obtain
$$
\begin{aligned}
   G_\beta^\prime(R) 
   &= \frac{-Q_l^\prime(R)}{\sigma_s^{\prime\prime}(R)+\beta\sigma_s^{\prime}(R)}\\
   &= \beta \frac{Q_l(R)}{\sigma_s^{\prime\prime}(R)+\beta\sigma_s^{\prime}(R)}+1\\
   &= -\beta G_\beta(R)+1.
\end{aligned}
$$
By the maximum principle, we have $\left.\frac{\partial F}{\partial r}\right|_{r=R}>0$, therefore with $\beta>0$ we have  $\left.\frac{\partial G_\beta}{\partial l}\right|_{r=R}< 0$.

Since $G_\beta(r)>0$ for $r\in(R_0,R)$, $F<0$ in $r\in(R_0,R)$. Thus, $\left.\frac{\partial F}{\partial r}\right|_{r=R_0}<0$ and $\left.\frac{\partial F}{\partial r}\right|_{r=R}>0$, \textit{i.e.}, $G_\beta^{\prime}(R_0; l)>0$ is a decreasing sequence in $l$, and $G_\beta^{\prime}(R ; l)$ is an increasing sequence in $l$.
\end{proof}
Similarly, notice that $$\Delta (p_{1}+(\mathcal{P}-\chi_\sigma) \sigma_{1})=0,$$ and by separation of variables, we have
$$
p_{1}+(\mathcal{P}-\chi_\sigma) \sigma_{1}=({D}_{1} r^{l} +{D}_{2} r^{-l}) \cos (l \theta),
$$
where ${D}_{1}$ and ${D}_{2}$ satisfy
$$
\left\{
\begin{aligned}
l\left(D_{1} R_{0}^{l-1}-D_{2}R_{0}^{-(l+1)}\right)&=\mathcal{P}
Q^\prime_l(R_0),\\
D_{1} R^{l}+D_2 R^{-l}&=\mathcal{G}^{-1} \frac{l^{2}-1}{R^{2}}-p_s^\prime(R)+\left(\mathcal{P}-\chi_{\sigma}\right)Q_l(R),
\end{aligned}\right.
$$
then we have
$$
\begin{aligned}  
D_{1}=&\frac{R_{0}^{l+1}}{R^{2 l}+R_{0}^{2 l}}\left(\frac{\mathcal{P}}{l}
Q^\prime_l(R_0)\right) +\frac{R^{l}}{R^{2 l}+R_{0}^{2 l}}\left(\mathcal{G}^{-1} \frac{l^{2}-1}{R^{2}}+(\mathcal{P}-\chi_\sigma)Q_l(R)-p_s^\prime(R)\right),
\\
D_{2}=&\frac{R_{0}^{l+1}R^{2l}}{R^{2 l}+R_{0}^{2 l}}\left(-\frac{\mathcal{P}}{l}
Q^\prime_l(R_0)\right)+\frac{R^{l}R_0^{2l}}{R^{2 l}+R_{0}^{2 l}}\left(\mathcal{G}^{-1} \frac{l^{2}-1}{R^{2}}+(\mathcal{P}-\chi_\sigma)Q_l(R)-p_s^\prime(R)\right).
\end{aligned}
$$

\begin{remark}
For $l=0$, we have $p_1+(\mathcal{P}-\chi_\sigma)\sigma_1=D$ (a constant) where $D=-\frac{\mathcal{G}^{-1}}{R^2}+(\mathcal{P-\chi_\sigma})Q_0(R)-p_s^\prime(R)$ by the boundary conditions of $p_1+(\mathcal{P}-\chi_\sigma)\sigma_1$ at $r=R$, which also imply $p_1=-\frac{\mathcal{G}^{-1}}{R^2}-p_s^\prime(R)$ and $\frac{\partial p_1}{\partial r}=-(\mathcal{P}-\chi_\sigma)\frac{\partial\sigma_1}{\partial r}=-(\mathcal{P}-\chi_\sigma)Q_0^\prime(R)$ for $l=0$.
\end{remark}

Thus
\begin{align}\notag
&D_1r^l+D_2r^{-l}\\
=&\ \notag 
\frac{\mathcal{P}}{l}\frac{R_0^{l+1}(r^l-R^{2l}r^{-l})}{R^{2 l}+R_{0}^{2 l}}
Q^\prime_l(R_0)\\
&\ +
\frac{R^l(r^l+R_0^{2l}r^{-l})}{R^{2 l}+R_{0}^{2 l}}
\left(\mathcal{G}^{-1} \frac{l^{2}-1}{R^{2}}+(\mathcal{P}-\chi_\sigma)Q_l(R)-p_s^\prime(R)\right),
\end{align}
and its derivative
\begin{align}\notag
&l(D_1r^{l-1}-D_2r^{-(l+1)})\\
=&\ \notag
\mathcal{P}\frac{R_0^{l+1}(r^{l-1}+R^{2l}r^{-(l+1)})}{R^{2 l}+R_{0}^{2 l}}
Q^\prime_l(R_0)
\\
&\ +
\frac{lR^l(r^{l-1}-R_0^{2l}r^{-(l+1)})}{R^{2 l}+R_{0}^{2 l}}
\left(\mathcal{G}^{-1} \frac{l^{2}-1}{R^{2}}+(\mathcal{P}-\chi_\sigma)Q_l(R)-p_s^\prime(R)\right).
\end{align}
Note that
$$
\lim_{R_0\rightarrow 0}
D_1r^l+D_2r^{-l}
=
\left(\mathcal{G}^{-1} \frac{l^{2}-1}{R^{2}}+(\mathcal{P}-\chi_\sigma)Q_l(R)-p_s^\prime(R)\right)
\left(\frac{r}{R}\right)^l
,
$$
and its derivative
$$
\lim_{R_0\rightarrow 0}
l(D_1r^{(l-1)}-D_2r^{-(l+1)})
=
\left(\mathcal{G}^{-1} \frac{l^{2}-1}{R^{2}}+(\mathcal{P}-\chi_\sigma)Q_l(R)-p_s^\prime(R)\right)
\left(\frac{r}{R}\right)^l
\frac{l}{r}
.
$$
\subsection{Justification for $\eqref{expansions}$}
\label{sect:Jus_exp}
In this subsection, we justify the validity of expansions in $\eqref{expansions}$ by showing that the $\mathcal{O}\left(\epsilon^{2}\right)$ terms are small. First, we introduce the following Banach space:
$$
\begin{aligned}
X^{l+\alpha}&=\left\{R_{1} \in C^{l+\alpha}: R_{1}\text{ is }2 \pi\text{-periodic}\right\}\\
X_{1}^{l+\alpha}&=\text{ closure of the linear space spanned}\\
&\text{by }\{\cos (j \theta), j=0,1,2, \ldots\}\text{ in }X^{l+\alpha}.
\end{aligned}
$$

Since we have included all modes, all even functions with period $2\pi$ can be expanded into a Fourier series, and the algebra property is automatically satisfied.
To derive rigorous estimates, we start with
the following lemma:
\begin{lemma}
If $R_{1} \in C^{3+\alpha}(\mathbb{R})$ and $(\sigma, p)$ is the solution of $\eqref{linpde}$, then
$$
\begin{aligned}
&\left\|\sigma-\sigma_{s}\right\|_{C^{3+\alpha}\left(\bar{\Omega}_{\varepsilon}\right)} \leq C|\varepsilon|\left\|R_{1}\right\|_{C^{3+\alpha}(\mathrm{R})}, \\
&\left\|p-p_{s}\right\|_{C^{1+\alpha}\left(\bar{\Omega}_{\varepsilon}\right)} \leq C|\varepsilon|\left\|R_{1}\right\|_{C^{3+\alpha}(\mathbb{R})},
\end{aligned}
$$
where $ {C}$ is a constant independent of $\varepsilon .$
\end{lemma}
Notice that even though we solved our equation for $\sigma_s$ on the domain $\Omega$, the solution $\sigma_s$ is given by an explicit formula and therefore is well defined in the entire space and the above estimates make sense.
\begin{proof}
First, we derive the equation of $\sigma-\sigma_{s}$ below
\begin{equation}\label{sig-sigs}
\left\{
\begin{array}{ccc}
\begin{aligned}
\Delta\left(\sigma-\sigma_{s}\right)-\left(\sigma-\sigma_{s}\right)&=0 & \text { in } \Omega_{\varepsilon} ,
\\ \sigma-\sigma_{s}&=0 & \text { on } \Gamma_{0},
\\ {
\frac{\partial\left(\sigma-\sigma_{s}\right)}{\partial \mathbf{n}}+\beta\left(\sigma-\sigma_{s}\right) }& {
=g_{1}+\beta g_{2} }& \text { on } \Gamma_{\varepsilon}.
\end{aligned}
\end{array}
\right.
\end{equation}

Clearly,
\begin{eqnarray*}
&&\hspace{-2em}   g_1+\beta g_2  \\
&=&\frac{\partial\sigma\left(R+\varepsilon R_{1},\theta\right)}{\partial\mathbf{n}}-\frac{\partial\sigma_s\left(R+\varepsilon R_{1}\right)}{\partial\mathbf{n}} +\beta\left[\sigma(R+\varepsilon R_{1},\theta)-\sigma_{s}(R+\varepsilon R_{1})\right] \\
&=&\frac{\partial\sigma_s\left(R \right)}{\partial r} -\frac{\partial\sigma_s\left(R+\varepsilon R_{1}\right)}{\partial\mathbf{n}} +\beta\left[\sigma_s(R)-\sigma_{s}(R+\varepsilon R_{1})\right] .
\end{eqnarray*}
Since $\sigma_s$ is given explicitly, after differentiating in $\theta$ two times, we find that the $C^{2+\alpha}$ norm of the right-hand side
of the above expression is clearly bounded by $ \|\varepsilon R_1\|_{C^{3+\alpha}}$.


The Schauder estimates then indicate that
$$
\left\|\sigma-\sigma_{s}\right\|_{C^{3+\alpha}\left(\tilde{\Omega}_{\varepsilon}\right)} \leq C|\varepsilon|\left\|R_{1}\right\|_{C^{3+\alpha}(\mathbb{R})} .
$$
Since $\Gamma_{\varepsilon} \in C^{3+\alpha}$, the constant $C$ is independent of $\varepsilon$.

Similarly, the equation of $p-p_{s}$ reads as
\begin{equation}\label{p-ps}
\left\{
\begin{array}{ccc}
\begin{aligned}-\Delta\left(p-p_{s}\right)&=\left(\mathcal{P}-\chi_\sigma\right)\left(\sigma-\sigma_{s}\right) & \text { in } \Omega_{\varepsilon}, \\ 
\frac{\partial\left(p-p_{s}\right)}{\partial \mathbf{n}}&=0 & \text { on } \Gamma_{0}, \\
p-p_{s}&=g_{3} & \text { on } \Gamma_{\varepsilon},
\end{aligned}
\end{array}
\right.
\end{equation}

where
$$
\begin{aligned}
g_{3} &=p\left(R+\varepsilon R_{1}\right)-p_{s}\left(R+\varepsilon R_{1}\right)\\
&=\left.\mathcal{G}^{-1} \kappa\right|_{r=R+\varepsilon R_{1}}-p_{s}\left(R+\varepsilon R_{1}\right) \\
&=\frac{\mathcal{G}^{-1}}{R}-\mathcal{G}^{-1} \frac{\varepsilon}{R^{2}}\left(R_{1}+R_{1 \theta \theta}\right)+\mathcal{O}\left(\varepsilon^{2}\right)-p_{s}\left(R+\varepsilon R_{1}\right) \\
&=p_{s}(R)-p_{s}\left(R+\varepsilon R_{1}\right)-\mathcal{G}^{-1} \frac{\varepsilon}{R^{2}}\left(R_{1}+R_{1 \theta \theta}\right)+\mathcal{O}\left(\varepsilon^{2}\right)\\
&=\varepsilon R_1 \frac{\mathcal{G}^{-1}}{R^2}-\mathcal{G}^{-1} \frac{\varepsilon}{R^{2}}\left(R_{1}+R_{1 \theta \theta}\right)+\mathcal{O}\left(\varepsilon^{2}\right)\\
&=- \varepsilon\frac{\mathcal{G}^{-1}}{R^{2}}R_{1 \theta \theta}+\mathcal{O}\left(\varepsilon^{2}\right).\\
\end{aligned}
$$
We differentiate the above equation along $\Gamma_{\varepsilon}$ and get
$$
\left\|p-p_{s}\right\|_{C^{1+\alpha}\left(\Gamma_{\varepsilon}\right)} \leq C|\varepsilon|\left\|R_{1}\right\|_{C^{3+\alpha}(\mathbb{R})} .
$$
The Schauder estimates imply
$$
\begin{aligned}
\left\|p-p_{s}\right\|_{C^{1+\alpha}\left(\bar{\Omega}_{\varepsilon}\right)}
&\leq C\left\|\sigma-\sigma_{s}\right\|_{C^{\alpha}\left(\bar{\Omega}_{\varepsilon}\right)}+C\left\|p-p_{s}\right\|_{C^{1+\alpha}\left(\Gamma_{\varepsilon}\right)}\\
&\leq C|\varepsilon|\left\|R_{1}\right\|_{C^{3+\alpha}(\mathbb{R})}.
\end{aligned}
$$
Due to the regularity of $\sigma_{s}$ and $p_{s}$ and $\Gamma_{\varepsilon} \in C^{3+\alpha}$, we conclude the constant $C$ is independent of $\varepsilon$.
\end{proof}

Next, we proceed to rigorously establish $\eqref{expansions}$. Notice that 
both $\sigma$ and $p$ are defined on $\Omega_{\varepsilon}$, but unlike $\sigma_s$, the first order expansion terms $\sigma_{1}$ and $p_{1}$ are defined on $\Omega$ only, we need to transform $\sigma_{1}$ and $p_{1}$ to $\Omega_{\varepsilon}$ by Hanzawa transformation $H_{\varepsilon}$,
$$
(r, \theta)=H_{\varepsilon}\left(r^{\prime}, \theta^{\prime}\right)=\left(r^{\prime}+\chi\left(r^{\prime}-R\right) \varepsilon R_{1}, \theta^{\prime}\right),
$$
where
$$
\chi \in C^{\infty}, \quad \chi(z)=\left\{\begin{array}{ll}
0 & \text { if }|z| \geq \frac{3}{4} \delta_{0}, \\
1 & \text { if }|z|<\frac{1}{4} \delta_{0},
\end{array}\left|\frac{d^{k} \chi}{d z^{k}}\right| \leq \frac{C}{\delta_{0}^{k}},\right.
$$
and $\delta_{0}>0$ is small. Noticing that $H_{\varepsilon}$ maps $\Omega$ onto $\Omega_{\varepsilon}$ but keeps the annulus $\left\{r: R_0 \leq r \leq R-\frac{3}{4} \delta_{0}\right\}$ fixed, we set

\begin{equation}\label{Hanzawa}
\tilde{\sigma}_{1}(r, \theta)=\sigma_{1}\left(H_{\varepsilon}^{-1}(r, \theta)\right), \quad \tilde{p}_{1}(r, \theta)=p_{1}\left(H_{\varepsilon}^{-1}(r, \theta)\right).    
\end{equation}

Then, we establish the following estimates.

\begin{theorem}
If $R_{1} \in C^{3+\alpha}(\mathbb{R}),(\sigma, p)$ is the solution of \eqref{linpde}, and $\left(\sigma_{1}, p_{1}\right)$ is defined as $\eqref{Hanzawa}$, then
$$
\begin{aligned}
\left\|\sigma-\sigma_{s}-\varepsilon \tilde{\sigma}_{1}\right\|_{C^{3+\alpha}\left(\bar{\Omega}_{\varepsilon}\right)} &\leq C|\varepsilon|^{2}\left\|R_{1}\right\|_{C^{3+\alpha}(\mathbb{R})} \\
\left\|p-p_{s}-\varepsilon \tilde{p}_{1}\right\|_{C^{1+\alpha}\left(\bar{\Omega}_{\varepsilon}\right)} &\leq C|\varepsilon|^{2}\left\|R_{1}\right\|_{C^{3+\alpha}(\mathbb{R})} .
\end{aligned}
$$
\end{theorem}
\begin{proof}
First, we compute the first and second derivatives of $\tilde{\sigma}_{1}$ with respect to both $r$ and $\theta$ :
$$
\begin{aligned}
\frac{\partial \tilde{\sigma}_{1}}{\partial r}&=\frac{\partial \sigma_{1}}{\partial r^{\prime}} \frac{\partial r^{\prime}}{\partial r}, \quad \frac{\partial \tilde{\sigma}_{1}}{\partial \theta}=\frac{\partial \sigma_{1}}{\partial r^{\prime}} \frac{\partial r^{\prime}}{\partial \theta}+\frac{\partial \sigma_{1}}{\partial \theta^{\prime}} \\
\frac{\partial^{2} \tilde{\sigma}_{1}}{\partial r^{2}}&=\frac{\partial^{2} \sigma_{1}}{\partial r^{2}}\left(\frac{\partial r^{\prime}}{\partial r}\right)^{2}+\frac{\partial \sigma_{1}}{\partial r^{\prime}} \frac{\partial^{2} r^{\prime}}{\partial r^{2}} \\
\frac{\partial^{2} \tilde{\sigma}_{1}}{\partial \theta^{2}}&=\frac{\partial^{2} \sigma_{1}}{\partial \theta^{\prime 2}}
+2 \frac{\partial^{2} \sigma_{1}}{\partial r^{\prime} \partial \theta^{\prime}} \frac{\partial r^{\prime}}{\partial \theta}
+\frac{\partial^{2} \sigma_{1}}{\partial r^{\prime 2}}\left(\frac{\partial r^{\prime}}{\partial \theta}\right)^{2}+\frac{\partial \sigma_{1}}{\partial r^{\prime}} \frac{\partial^{2} r^{\prime}}{\partial \theta^{2}}
\end{aligned}
$$
where the derivatives of $r^{\prime}$ is derived by the Hanzawa transformation. In fact, the first derivatives are
$$
\begin{aligned}
&1=\frac{\partial r^{\prime}}{\partial r}+\varepsilon R_{1} \chi^{\prime}\left(r^{\prime}-R\right) \frac{\partial r^{\prime}}{\partial r}, \\
&0=\frac{\partial r^{\prime}}{\partial \theta}+\varepsilon R_{1} \chi^{\prime}\left(r^{\prime}-R\right) \frac{\partial r^{\prime}}{\partial \theta}+\varepsilon \chi\left(r^{\prime}-R\right) R_{1 \theta};
\end{aligned}
$$
thus,
$$
\frac{\partial r^{\prime}}{\partial r}=\frac{1}{1+\varepsilon R_{1} \chi^{\prime}\left(r^{\prime}-R\right)} \text { and } \frac{\partial r^{\prime}}{\partial \theta}=-\frac{\varepsilon \chi\left(r^{\prime}-R\right) R_{1 \theta}}{1+\varepsilon R_{1} \chi^{\prime}\left(r^{\prime}-R\right)} .
$$
Similarly, we obtain the second derivatives below
$$
\begin{aligned}
\frac{\partial^{2} r^{\prime}}{\partial r^{2}}=&-\frac{\varepsilon R_{1} \chi^{\prime \prime}\left(r^{\prime}-R\right)}{\left(1+\varepsilon R_{1} \chi^{\prime}\left(r^{\prime}-R\right)\right)^{2}} \frac{\partial r^{\prime}}{\partial r}=-\frac{\varepsilon R_{1} \chi^{\prime \prime}\left(r^{\prime}-R\right)}{\left(1+\varepsilon R_{1} \chi^{\prime}\left(r^{\prime}-R\right)\right)^{3}}, \\
\frac{\partial^{2} r^{\prime}}{\partial \theta^{2}}=&-\frac{\varepsilon \chi\left(r^{\prime}-R\right) R_{1 \theta \theta}}{1+\varepsilon R_{1} \chi^{\prime}\left(r^{\prime}-R\right)}+2 \frac{\varepsilon^{2} \chi\left(r^{\prime}-R\right) \chi^{\prime}\left(r^{\prime}-R\right) R_{1 \theta}^{2}}{\left(1+\varepsilon R_{1} \chi^{\prime}\left(r^{\prime}-R\right)\right)^{2}} \\
&-\frac{\left(\chi\left(r^{\prime}-R\right) \varepsilon R_{1 \theta}\right)^{2} \chi^{\prime \prime}\left(r^{\prime}-R\right) \varepsilon R_{1}}{\left(1+\varepsilon R_{1} \chi^{\prime}\left(r^{\prime}-R\right)\right)^{3}}.
\end{aligned}
$$

Next, we consider the estimate of $\phi=\sigma-\sigma_{s}-\varepsilon \tilde{\sigma}_{1}$, which satisfies
$$
\left\{
\begin{array}{ccc}
\begin{aligned}
\Delta \phi- \phi&=\varepsilon^{2} \tilde{f} & \text { in } \Omega_{\varepsilon}, \\ 
\phi&=0 & \text { on } \Gamma_{0}, \\
\frac{\partial \phi}{\partial \bf n}+\beta\phi&=g_4+\beta g_5 & \text { on } \Gamma_{\varepsilon},
\end{aligned}
\end{array}\right.
$$
where $\tilde{f}$ depends on various terms of the Hanzawa transformation above and involves up to second-order derivatives of $R_{1}$ and $\sigma_{1}$. By applying the Schauder estimate to $\eqref{linearized}$, we know $\sigma_{1} \in C^{3+\alpha}$ and
$$
\left\|\tilde{f}\right\|_{C^{1+\alpha}\left(\Omega_{\varepsilon}\right)} \leq C\left\| R_{1} \right\|_{C^{3+\alpha}(\mathbb{R})} .
$$

On the boundary $\Gamma_{\varepsilon}$, 
clearly we have,
\begin{eqnarray*}
g_4+\beta g_5
&=&\frac{\partial\sigma\left(R+\varepsilon R_{1},\theta\right)}{\partial\mathbf{n}}-\frac{\partial\sigma_s\left(R+\varepsilon R_{1}\right)}{\partial\mathbf{n}} 
-\varepsilon\left.\frac{\partial\tilde\sigma_1}{\partial\mathbf{n}}\right|_{r=R+\varepsilon R_{1}}\\
&&+\beta\left[\sigma(R+\varepsilon R_{1},\theta)-\sigma_{s}(R+\varepsilon R_{1})
-\varepsilon \left.\tilde{\sigma}_{1}\right|_{r=R+\varepsilon R_{1}}
\right] 
\\
&=&\frac{\partial\sigma_s\left(R \right)}{\partial r} -\frac{\partial\sigma_s\left(R+\varepsilon R_{1}\right)}{\partial\mathbf{n}} +\beta\left[\sigma_s(R)-\sigma_{s}(R+\varepsilon R_{1})\right]\\
&&-\varepsilon\left.\frac{\partial\tilde\sigma_1}{\partial\mathbf{n}}\right|_{r=R+\varepsilon R_{1}}
-\varepsilon \beta\left.\tilde{\sigma}_{1}\right|_{r=R+\varepsilon R_{1}}
.
\end{eqnarray*}
Substituting the boundary condition of $\sigma_1$ from $\eqref{linearized}$ into the above equation, we obtain an expression involving $\sigma_s$ 
and $\varepsilon R_1$ only. Since $\sigma_s$ is given explicitly, after differentiating in $\theta$ two times, we find that the $C^{2+\alpha}$ norm of the right-hand side
of the above expression is clearly bounded by $ \|\varepsilon^2 R_1\|_{C^{3+\alpha}}$.

\void{\green
by \eqref{g1} we have
$$
\begin{aligned}
g_4 &=\frac{\partial\sigma\left(R+\varepsilon R_{1}\right)}{\partial\mathbf{n}}-\frac{\partial\sigma_s\left(R+\varepsilon R_{1}\right)}{\partial\mathbf{n}}-\varepsilon\left.\frac{\partial\tilde\sigma_1}{\partial\mathbf{n}}\right|_{r=R+\varepsilon R_{1}} \\
&=\varepsilon \left.\frac{\partial\sigma_1}{\partial r}\right|_{r=R}-\varepsilon \left.\frac{\partial\tilde\sigma_1}{\partial r}\right|_{r=R+\varepsilon R_1}+\mathcal{O}\left(\varepsilon^{2}\right)\\
&=\mathcal{O}(\varepsilon^2)R_1,
\end{aligned}
$$
and by $\eqref{g2}$ we have
$$
\begin{aligned}
g_5 &=\sigma\left(R+\varepsilon R_{1}\right)-\sigma_{s}\left(R+\varepsilon R_{1}\right)-\varepsilon \left.\tilde{\sigma}_{1}\right|_{r=R+\varepsilon R_{1}} \\
&=\varepsilon \left.\sigma_1\right|_{r=R}-\varepsilon \left.\tilde{\sigma}_{1}\right|_{r=R+\varepsilon R_{1}}\\
&=\mathcal{O}\left(\varepsilon^{2}\right) R_{1}.
\end{aligned}
$$
}

By the Schauder theory, we obtain
$$
\left\|\sigma-\sigma_{s}-\varepsilon \tilde{\sigma}_{1}\right\|_{C^{3+\alpha}\left(\bar{\Omega}_{\varepsilon}\right)} \leq C|\varepsilon|^{2}\left\|R_{1}\right\|_{C^{3+\alpha}(\mathbb{R})}.
$$
Similarly, we write the equation of $\psi=p-p_{s}-\varepsilon \tilde{p}_{1}$ as follows:
\begin{equation}
\left\{
\begin{array}{ccc}
\begin{aligned}-\Delta \psi&=(\mathcal{P}-\chi_\sigma) \phi+\varepsilon^{2} \tilde{k} & \text { in } \Omega_{\varepsilon}, \\
\frac{\partial \psi}{\partial r}&=0 & \text { on } \Gamma_{0},\\
\psi&=f & \text { on } \Gamma_{\varepsilon}, 
\end{aligned}
\end{array}
\right.
\end{equation}
where $\tilde{k}$ is based on various term of the Hanzawa transformation above and follows:
$$
\|\tilde{k}\|_{C^{1+\alpha}\left(\bar{\Omega}_{\varepsilon}\right)} \leq C\left\|R_{1}\right\|_{C^{3+\alpha}(\mathbb{R})} .
$$
Since
$$
f=p\left(R+\varepsilon R_{1} , \theta\right)-p_{s}\left(R+\varepsilon R_{1}\right)-\varepsilon \tilde{p}_{1}\left(R+\varepsilon R_{1}, \theta\right),
$$
we have
$$
\|f\|_{C^{1+\alpha}(\mathbb{R})} \leq C|\varepsilon|^{2}\left\|R_{1}\right\|_{C^{3+\alpha}(\mathbb{R})} .
$$
Therefore, by the Schauder estimate, we conclude
$$
\left\|p-p_{s}-\varepsilon \tilde{p}_{1}\right\|_{C^{1+\alpha}\left(\bar{\Omega}_{\varepsilon}\right)} \leq C|\varepsilon|^{2}\left\|R_{1}\right\|_{C^{3+\alpha}(\mathbb{R})} .
$$
\end{proof}

\subsection{Bifurcation analysis}
We consider the nonlinear function $F$ defined in \eqref{bifurcation} by expanding $\frac{\partial p}{\partial r}$ on $\Gamma_{\varepsilon}$, namely,
$$
\begin{aligned}
F&\left(R_{1}, \mathcal{P}\right)\\&=-\left.\frac{\partial p}{\partial r} \right|_{\Gamma_{\varepsilon}}+\chi_\sigma\left.\frac{\partial \sigma}{\partial r} \right|_{\Gamma_{\varepsilon}}\\
&=-\varepsilon\left(\left.\frac{\partial p_1}{\partial r}\right|_{r=R}+p_s^{\prime\prime}(R)R_1-\chi_\sigma\left(\left.\frac{\partial\sigma_1}{\partial r}\right|_{r=R}+\sigma_s^{\prime\prime}(R)R_1\right)\right)+\mathcal{O}\left(|\varepsilon|^{2}\right).
\end{aligned}
$$
Thus, $F$ maps $\left(R_{1},\mathcal{P}\right)$ from $X^{l+3+\alpha}$ to $X^{l+\alpha}$ and is bounded for any $l \geq 0$. Furthermore, $F$ is Fréchet differentiable and the Fréchet derivative at $(0, \mathcal{P})$ is given by
$$
\left[\frac{\partial F}{\partial R_{1}}(0, \mathcal{P})\right] \cos (l \theta)
=\chi_\sigma\left(\left.\frac{\partial\sigma_1}{\partial r}\right|_{r=R}+\sigma_s^{\prime\prime}(R)R_1\right)-\left(\left.\frac{\partial p_1}{\partial r}\right|_{r=R}+p_s^{\prime\prime}(R)R_1\right).
$$
Then, the bifurcation condition becomes
$$
F(\mathcal{P}):=
\left.\frac{\partial p_1}{\partial r}\right|_{r=R}+p_s^{\prime\prime}(R)R_1
-
\chi_\sigma\left(\left.\frac{\partial\sigma_1}{\partial r}\right|_{r=R}+\sigma_s^{\prime\prime}(R)R_1\right)
=0.
$$
Since for $l\ge 0$,
$$
\begin{aligned}
\left.\frac{\partial p_{1}}{\partial r}\right|_{r=R}&=\left[-(\mathcal{P}-\chi_\sigma) Q_l^{\prime}(R)+l {D}_{1} R^{l-1}-l {D}_{2} R^{-l-1}\right]\cos (l\theta),\\
\left.\frac{\partial \sigma_{1}}{\partial r}\right|_{r=R}&=Q_l^{\prime}(R)\cos (l\theta),
\end{aligned}
$$
and
$$
\begin{aligned}
p_{s}^{\prime\prime}(R)+\frac{1}{R}p_{s}^{\prime}(R)&=\mathcal{P}\mathcal{A}- (\mathcal{P}-\chi_\sigma)\sigma_{s}(R),\\
\sigma_{s}^{\prime\prime}(R)+\frac{1}{R}\sigma_{s}^{\prime}(R)&=\sigma_s(R),\\
p_s^\prime(R)&=\chi_\sigma \sigma_s^\prime(R),
\end{aligned}
$$
we obtain
$$\small
\begin{aligned}
F(\mathcal{P})
=&\ \mathcal{P}(\mathcal{A}-\sigma_s(R)-Q_l^\prime(R))+l {D}_{1} R^{l-1}-l {D}_{2} R^{-l-1} \\
=&\ \mathcal{P}\left(2\frac{R\sigma_s^\prime(R)-R_0\sigma_s^\prime(R_0) }{R^2-R_0^2}-\sigma_s(R)-Q_l^\prime(R)+\frac{2R_{0}^{l+1}R^{l-1}}{R^{2 l}+R_{0}^{2 l}}Q_l^\prime(R_0)\right)\\
\\
&\ +\frac{l(R^{2l-1}-R_0^{2l}R^{-1})}{R^{2 l}+R_{0}^{2 l}}\left(\mathcal{G}^{-1} \frac{l^{2}-1}{R^{2}}+(\mathcal{P}-\chi_\sigma)Q_l(R)-p_s^\prime(R)\right)\\
=
&\ \mathcal{P}\left(\frac{2}{R}\frac{\sigma_s^\prime(R)-\frac{R_0}{R}\sigma_s^\prime(R_0) }{1-\left(\frac{R_0}{R}\right)^2}-\sigma_s(R)-Q_l^\prime(R)+\frac{l}{R}\frac{1-\left(\frac{R_0}{R}\right)^{2l}}{1 +\left(\frac{R_{0}}{R}\right)^{2l}}Q_l(R)\rule{0pt}{6ex}\right.\\  &\hspace{2em} \left. +\frac{Q_l^\prime(R_0)}{\frac{1}{2}\left(\left(\frac{R}{R_0}\right)^{l+1}+\left(\frac{R_0}{R}\right)^{l-1}\right)}\right)\\
&\ +\frac{l}{R}\frac{1-\left(\frac{R_0}{R}\right)^{2l}}{1 +\left(\frac{R_{0}}{R}\right)^{2 l}}\left(\mathcal{G}^{-1} \frac{l^{2}-1}{R^{2}}-\chi_\sigma(Q_l(R)+\sigma_s^\prime(R))\right)
\\
=
&\ 0 .
\end{aligned}
$$\normalsize
Therefore, the formula of $\mathcal{P}_{l}$ for bifurcation points is
\begin{equation}\label{bifurcation_point}
    \mathcal{P}_{l}=\frac{L_{1}(l, R, R_0)}{L_{2}(l, R, R_0)},
\end{equation}
where
$$
\begin{aligned}
L_{1}(l, R, R_0)=&\frac{l}{R}\overbrace{\frac{1-\left(\frac{R_0}{R}\right)^{2l}}{ 1+
\left(\frac{R_{0}}{R}\right)^{2 l}}}^\text{Necrosis (I)}\left(\overbrace{\mathcal{G}^{-1} \frac{l^{2}-1}{R^{2}}}^\text{Surface tension}-\overbrace{\chi_\sigma(Q_l(R)+\sigma_s^\prime(R))}^\text{Chemotaxis}\right),
\end{aligned}
$$
and
$$
\begin{aligned}
L_{2}&(l, R, R_0)\\=&\overbrace{\sigma_s(R)}^\text{Nutrient at boundary}-\overbrace{\frac{2}{R}\frac{\sigma_s^\prime(R)-\frac{R_0}{R}\sigma_s^\prime(R_0) }{1-\left(\frac{R_0}{R}\right)^2}}^\text{Apoptosis}\\
&+\underbrace{Q_l^\prime(R)-\frac{l}{R}\overbrace{\frac{1-\left(\frac{R_0}{R}\right)^{2l}}{ 1+
\left(\frac{R_{0}}{R}\right)^{2 l}}}^{\text{Necrosis (I)}}Q_l(R)
-
\overbrace{\frac{1}{\frac{1}{2}\left(\left(\frac{R}{R_0}\right)^{l+1}+\left(\frac{R_0}{R}\right)^{l-1}\right)}}^\text{Necrosis (II)}Q_l^\prime(R_0)}_{\Lambda\text{: Necrosis and Nutrient perturbation}}.
\end{aligned}
$$

\begin{remark}
    For $l=0$, $F(\mathcal{P})=\mathcal{P}(\mathcal{A}-\sigma_s(R)-Q_0^\prime(R))=0$, which implies $\mathcal{P}_0=0$.
\end{remark}

Note that
\begin{align}
\lim_{R_0\rightarrow 0}\mathcal{P}_l
&=
\frac{
\overbrace{\mathcal{G}^{-1}\frac{l^3-l}{R^3}}^\text{Surface tension}
-\overbrace{\chi_\sigma\lim_{R_0\rightarrow 0}(Q_l(R)+\sigma_s^\prime(R))
\frac{l}{R}}^\text{Chemotaxis}}
{\displaystyle\lim_{R_0\rightarrow 0}\left(
\underbrace{\sigma_s(R)}_\text{Nutrient at boundary}
-\underbrace{\frac{2}{R}\sigma^{\prime}_s(R)}_\text{Apoptosis}
+\underbrace{Q^\prime_l(R)-\frac{l}{R}Q_l(R)}_\text{Nutrient perturbation}\right)
},
\\
\lim_{\substack{\beta\rightarrow\infty\\R_0\rightarrow 0}}\mathcal{P}_l
&=
\frac{\overbrace{\mathcal{G}^{-1}\frac{l^3-l}{R^3}}^\text{Surface tension}}{\underbrace{1}_\text{Nutrient at boundary}-\underbrace{\frac{2}{R}\frac{I_1(R)}{I_0(R)}}_\text{Apoptosis}-
\underbrace{\frac{I_1(R)}{I_0(R)}\left(\frac{I_l^{\prime}(R)}{I_l(R)}-\frac{l}{R}  \right)}_\text{Nutrient perturbation}},\label{Pllim}
\end{align}
where \eqref{Pllim} recovers the result in \cite{friedman2001existence}. We found that it is independent of $\chi_\sigma$ since $\displaystyle\lim_{\beta\rightarrow\infty,R_0\rightarrow 0}\left(Q_l(R)+\sigma^\prime(R)\right)=-\frac{I_1(R)}{I_0(R)}+\frac{I_1(R)}{I_0(R)}=0$.
\begin{remark}
\begin{enumerate}
    \item It is clear that Necrosis (I) is in (0,1) and increasing in $l$. Necrosis (II) is also in (0,1) from the arithmetic-geometric mean inequality. To see the monotonicity of Necrosis (II) it is sufficient to check $g^\prime(l)>0$, where $g(l)=a^{l+1}+a^{1-l},a=R/R_0>1$. It is true since $g^{\prime}(l)=a^{1-l}\ln a (a^{2l}-1)>0$.
    \item 
    The monotonicity of $L_{2}(l, R, R_0)$ is summarized in the following Lemma $\ref{lemma:L2}$. Here we only consider $R_0$ in a neighborhood of $R$ with the assumption 
$\sigma_s(R)-\mathcal{A}>0$, where $\mathcal{A}=\frac{2}{R}\frac{\sigma_s^\prime(R)-\frac{R_0}{R}\sigma_s^\prime(R_0)}{1-\left(\frac{R_0}{R}\right)^2}$, which assumes the nutrient level at the tumor boundary is greater than the apoptosis rate.
    \item The only term that may significantly change the monotonicity is the chemotaxis in $L_1$, as we verified numerically in Fig. \ref{fig:mono} when $\chi_\sigma$ is increasing from 1 to 100. We remark that this effect is enhanced when we take a smaller value of $\beta$.
\end{enumerate}
\end{remark}

\begin{lemma}\label{lemma:L2}
For given $R>0$, $R_0$ is in a neighborhood of $R$, namely, $R-\varepsilon$ for a small $\varepsilon$, $L_2(l,R,R_0)>0$ is increasing with respect to $l$ under the assumption 
$\sigma_s(R)-\mathcal{A}>0$, where $\mathcal{A}=\frac{2}{R}\frac{\sigma_s^\prime(R)-\frac{R_0}{R}\sigma_s^\prime(R_0)}{1-\left(\frac{R_0}{R}\right)^2}$.
\end{lemma}
\begin{proof}
Recall
$$
L_{2}(l,R, R_0)= \sigma_s(R)-\frac{2}{R}\frac{\sigma_s^\prime(R)-\frac{R_0}{R}\sigma_s^\prime(R_0) }{1-\left(\frac{R_0}{R}\right)^2}
-(\sigma_s^{\prime\prime}(R)+\beta\sigma_s^\prime(R))
f(l).
$$
where
$$
f(l)=G_\beta^\prime(R)-\frac{l}{R}\frac{1-\left(\frac{R_0}{R}\right)^{2l}}{ 1+
\left(\frac{R_{0}}{R}\right)^{2 l}}G_\beta(R)
-
\frac{1}{\frac{1}{2}\left(\left(\frac{R}{R_0}\right)^{l+1}+\left(\frac{R_0}{R}\right)^{l-1}\right)}G_\beta^\prime(R_0) .
$$
Since $R_0=R-\varepsilon$, we have
$$
f(l)=
G_\beta^\prime(R)-G_\beta^\prime(R_0)
+\mathcal{O}\left(\varepsilon\right).
$$
From Lemma \ref{lemma:G'} we have $f(l)$, thus also for $L_{2}(l,R,R_0)$, increases with respect to $l$. 
Next, we prove that $L_{2}(l,R,R_0)>0$
when $\varepsilon$ is small and expand $L_{2}(l, R, R_0)$ in terms of $\varepsilon$
$$
L_{2}(l, R, R_0) =\sigma_s(R)-\mathcal{A}
+\mathcal{O}\left(\varepsilon\right).
$$
Since $\sigma_s(R)-\mathcal{A}>0$ by assumption, we have $L_{2}(l, R_0, R)>0$ for a small $\varepsilon$.
\end{proof}

\begin{figure}
    \centering
    \includegraphics[width=0.8\textwidth]{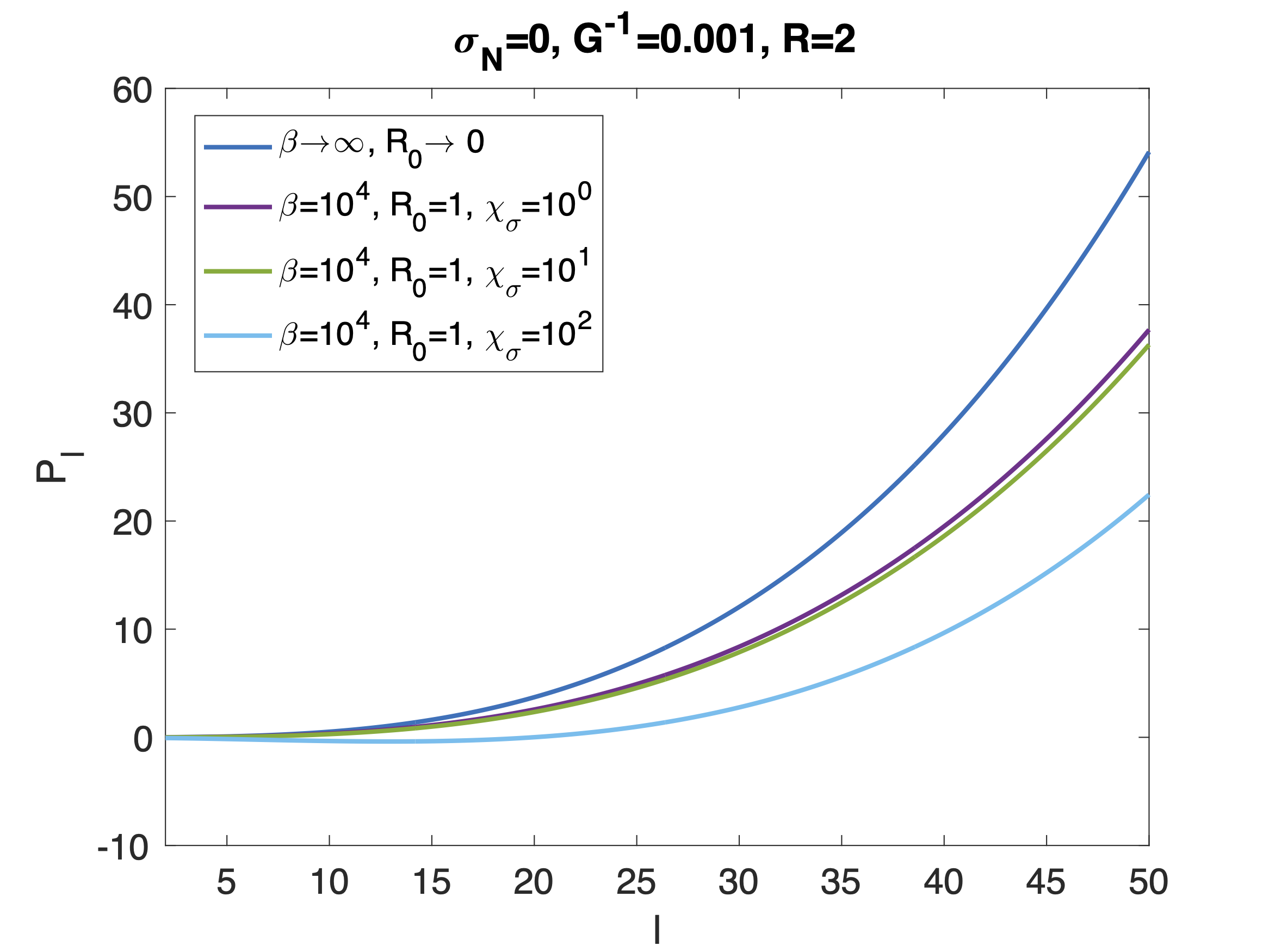}
    \caption{Plot of the effect of chemotaxis on monotonicity. The top curve is the limiting case and other curves are the relaxed cases with $\beta=10^4,R_0=1$ as the vascularized and necrotic case we considered in this paper. The monotonicity is lost as $\chi_\sigma$ is increasing.}
    \label{fig:mono}
\end{figure}

Then, we prove that $\mathcal{P}_{l}$ in \eqref{bifurcation} is a bifurcation point by verifying the following Crandall-Rabinowitz theorem.

\begin{theorem}
Let $X, Y$ be real Banach spaces and $F(x, \mu)$ a $C^{p}$ map, $p \geq 3$, of a neighborhood $\left(0, \mu_{0}\right)$ in $X \times \mathbb{R}$ into Y. Suppose
\begin{enumerate}
  \item $F(0, \mu)=0$ for all $\mu$ in a neighborhood of $\mu_{0}$,

  \item $\operatorname{Ker}F_{x}\left(0, \mu_{0}\right)$ is one dimensional space, spanned by $x_{0}$,

  \item $\operatorname{Im}F_{x}\left(0, \mu_{0}\right)=Y_{1}$ has codimension 1 ,

  \item $F_{\mu x}\left(0, \mu_{0}\right) \notin Y_{1}$.
\end{enumerate}
Then, $\left(0, \mu_{0}\right)$ is a bifurcation point of the equation $F(x, \mu)=0$ in the following sense: In a neighborhood of $\left(0, \mu_{0}\right)$, the set of solutions of $F(x, \mu)=0$ consists of two $\mathrm{C}^{p-2}$ smooth curves $\mathcal{C}_{1}$ and $\mathcal{C}_{2}$ which intersect only at the point $\left(0, \mu_{0}\right) .$ Moreover, $\mathcal{C}_{1}$ is the curve $\left(0, \mu_{0}\right)$ and $\mathcal{C}_{2}$ can be parameterized as follows:

$\mathcal{C}_{2}:(x(\varepsilon), \mu(\varepsilon)),|\varepsilon|$ small, $(x(0), \mu(0))=\left(0, \mu_{0}\right), x^{\prime}(0)=x_{0}$. 
\end{theorem}

\noindent \textbf{Verification.}
Notice that we have computed explicitly the first order Frech\'et derivative, which is clearly continuous. The argument actually shows that 
the differentiablility is eventually reduced to the regularity of the corresponding PDEs as shown in Sect. \ref{sect:Jus_exp}, and $\frac{\partial F}{\partial x}$ $\left(\text{or }\frac{\partial F}{\partial \mu}\right)$ is obtained by solving a linearized problem about $(x,\mu)$ with respect to $x$ (or $\mu$).
These explicit formulas are utilized to establish conditions (1)--(4) in the Crandall-Rabinowitz theorem. If we are just interested in differentiability, we can repeat the process
by using Schauder estimates, and we can actually obtain differentiability of $F(x, \mu)$ to any order, thus $F(x,\mu)$ is $C^{p}$. 
The structure of our PDE system guarantees that $F$ maps even $2\pi$-periodic functions to even $2\pi$-periodic functions,
and then the regularity implies that $F$ maps $X\times \mathbb R$ into $Y$.

Next, we choose the Banach spaces $X=X_{1}^{3+\alpha}$, $Y=X_{1}^{\alpha}, x=R_{1}$, and $\mu=\mathcal{P}$, then have
$\left[F_{R_{1}}(0, \mathcal{P})\right] \cos (l \theta)=\left(L_{1}(l, R, R_0)-\mathcal{P} L_{2}(l, R, R_0)\right) \cos (l \theta) .$
Thus, the kernel space satisfies
$$
\operatorname{ker}\left[F_{R_{1}}(0, \mathcal{P})\right]=\operatorname{span}\{\cos (l \theta)\} \quad \text { if } \mathcal{P}=\mathcal{P}_{l}
$$
and
$$
\operatorname{ker}\left[F_{R_{1}}(0, \mathcal{P})\right]=0 \quad \text { if } \mathcal{P} \neq \mathcal{P}_{1}, \mathcal{P}_{2}, \ldots ,
$$
which implies that $\operatorname{dim}\left(\operatorname{ker}\left[F_{R_{1}}(0, \mathcal{P})\right]\right)=1$. Moreover, since that $\operatorname{Im}\left[F_{R_{1}}\left(0, \mathcal{P}_{l}\right)\right] \oplus\{\cos (l \theta)\}$ is the whole space, we have $\operatorname{codim}\left(\operatorname{Im}\left[F_{R_{1}}\left(0, \mathcal{P}_{l}\right)\right]\right)=1$. Finally, by differentiating with respect to $\mathcal{P}$, we obtain
$$\left.\left[F_{R_{1} \mathcal{P}}(0, \mathcal{P})\right] \cos (l \theta)=-L_{2}(l, R, R_0)\right) \cos (l \theta) \notin \operatorname{Im}\left[F_{R_{1}}\left(0, \mathcal{P}_{l}\right)\right] .$$
Thus, all the assumptions in the Crandall-Rabinowitz theorem are satisfied.

\section*{Acknowledgments}
We thank the anonymous reviewer for providing detailed comments and suggestions that helped us to improve the paper. ML acknowledges National Institutes of Health for partial support through grant nos. 1U54CA217378-01A1 for a National Center in Cancer Systems Biology at UC Irvine, the support from DMS-1763272 and the Simons Foundation (594598QN) for an NSF-Simons Center for Multiscale Cell Fate Research and the support from NSF grant DMS-1953410. WH is supported by  the National Science Foundation (NSF) grant DMS-2052685. SL acknowledges the support from the NSF, Division of Mathematical Sciences grant DMS-1720420 and ECCS-1307625.

\section*{Compliance with Ethical Standards}
All authors state that there is no conflict of interest.

 \bibliographystyle{elsarticle-num} 
 \bibliography{cas-refs}





\end{document}